\documentclass[12pt,reqno]{amsart} 
\usepackage{amssymb,amscd,url}
\usepackage{constants} 
\usepackage{xcolor}    
\usepackage{bbm}
\usepackage{upref}     

\begin{document}

\allowdisplaybreaks



\title[An Elliptic Curve Function Field Lehmer Estimate]
      {A Lehmer-Type Lower Bound for the Canonical Height on Elliptic Curves Over Function Fields}
\date{\today}

\author[Joseph H. Silverman]{Joseph H. Silverman}
\email{joseph\_silverman@brown.edu}
\address{Mathematics Department, Box 1917
  Brown University, Providence, RI 02912 USA.
  OrcID  0000-0003-3887-3248. MR Author ID 162205.}




\hyphenation{ca-non-i-cal semi-abel-ian}


\newtheorem{theorem}{Theorem}
\newtheorem{lemma}[theorem]{Lemma}
\newtheorem{sublemma}[theorem]{Sublemma}
\newtheorem{conjecture}[theorem]{Conjecture}
\newtheorem{proposition}[theorem]{Proposition}
\newtheorem{corollary}[theorem]{Corollary}
\newtheorem*{claim}{Claim}

\theoremstyle{definition}
\newtheorem{definition}[theorem]{Definition}
\newtheorem{example}[theorem]{Example}
\newtheorem{remark}[theorem]{Remark}
\newtheorem{question}[theorem]{Question}

\theoremstyle{remark}
\newtheorem*{acknowledgement}{Acknowledgements}


\newenvironment{notation}[0]{%
  \begin{list}%
    {}%
    {\setlength{\itemindent}{0pt}
     \setlength{\labelwidth}{4\parindent}
     \setlength{\labelsep}{\parindent}
     \setlength{\leftmargin}{5\parindent}
     \setlength{\itemsep}{0pt}
     }%
   }%
  {\end{list}}

\newenvironment{parts}[0]{%
  \begin{list}{}%
    {\setlength{\itemindent}{0pt}
     \setlength{\labelwidth}{1.5\parindent}
     \setlength{\labelsep}{.5\parindent}
     \setlength{\leftmargin}{2\parindent}
     \setlength{\itemsep}{0pt}
     }%
   }%
  {\end{list}}
\newcommand{\Part}[1]{\item[\upshape#1]}

\def\Case#1#2{%
\paragraph{\textbf{\boldmath Case #1: #2.}}\hfil\break\ignorespaces}

\renewcommand{\a}{\alpha}
\newcommand{\bfalpha}{{\boldsymbol\alpha}}
\renewcommand{\b}{\beta}
\newcommand{\bfbeta}{{\boldsymbol\beta}}
\newcommand{\g}{\gamma}
\newcommand{\bfgamma}{{\boldsymbol\gamma}}
\renewcommand{\d}{\delta}
\newcommand{\e}{\epsilon}
\newcommand{\f}{\varphi}
\newcommand{\bfphi}{{\boldsymbol{\f}}}
\renewcommand{\l}{\lambda}
\renewcommand{\k}{\kappa}
\newcommand{\lhat}{\hat\lambda}
\newcommand{\m}{\mu}
\newcommand{\bfmu}{{\boldsymbol{\mu}}}
\renewcommand{\o}{\omega}
\renewcommand{\r}{\rho}
\newcommand{\rbar}{{\bar\rho}}
\newcommand{\s}{\sigma}
\newcommand{\sbar}{{\bar\sigma}}
\renewcommand{\t}{\tau}
\newcommand{\z}{\zeta}

\newcommand{\D}{\Delta}
\newcommand{\G}{\Gamma}
\newcommand{\F}{\Phi}
\renewcommand{\L}{\Lambda}

\newcommand{\ga}{{\mathfrak{a}}}
\newcommand{\gb}{{\mathfrak{b}}}
\newcommand{\gn}{{\mathfrak{n}}}
\newcommand{\gp}{{\mathfrak{p}}}
\newcommand{\gq}{{\mathfrak{q}}}
\newcommand{\gA}{{\mathfrak{A}}}
\newcommand{\gB}{{\mathfrak{B}}}
\newcommand{\gC}{{\mathfrak{C}}}
\newcommand{\gD}{{\mathfrak{D}}}
\newcommand{\gP}{{\mathfrak{P}}}
\newcommand{\gS}{{\mathfrak{S}}}

\newcommand{\Abar}{{\bar A}}
\newcommand{\Ebar}{{\bar E}}
\newcommand{\kbar}{{\bar k}}
\newcommand{\Kbar}{{\bar K}}
\newcommand{\Pbar}{{\bar P}}
\newcommand{\Sbar}{{\bar S}}
\newcommand{\Tbar}{{\bar T}}

\newcommand{\Acal}{{\mathcal A}}
\newcommand{\Bcal}{{\mathcal B}}
\newcommand{\Ccal}{{\mathcal C}}
\newcommand{\Dcal}{{\mathcal D}}
\newcommand{\Ecal}{{\mathcal E}}
\newcommand{\Fcal}{{\mathcal F}}
\newcommand{\Gcal}{{\mathcal G}}
\newcommand{\Hcal}{{\mathcal H}}
\newcommand{\Ical}{{\mathcal I}}
\newcommand{\Jcal}{{\mathcal J}}
\newcommand{\Kcal}{{\mathcal K}}
\newcommand{\Lcal}{{\mathcal L}}
\newcommand{\Mcal}{{\mathcal M}}
\newcommand{\Ncal}{{\mathcal N}}
\newcommand{\Ocal}{{\mathcal O}}
\newcommand{\Pcal}{{\mathcal P}}
\newcommand{\Qcal}{{\mathcal Q}}
\newcommand{\Rcal}{{\mathcal R}}
\newcommand{\Scal}{{\mathcal S}}
\newcommand{\Tcal}{{\mathcal T}}
\newcommand{\Ucal}{{\mathcal U}}
\newcommand{\Vcal}{{\mathcal V}}
\newcommand{\Wcal}{{\mathcal W}}
\newcommand{\Xcal}{{\mathcal X}}
\newcommand{\Ycal}{{\mathcal Y}}
\newcommand{\Zcal}{{\mathcal Z}}

\renewcommand{\AA}{{\mathbb{A}}}
\newcommand{\BB}{{\mathbb{B}}}
\newcommand{\CC}{{\mathbb{C}}}
\newcommand{\FF}{{\mathbb{F}}}
\newcommand{\FFbar}{{\overline{\mathbb{F}}}}
\newcommand{\GG}{{\mathbb{G}}}
\newcommand{\KK}{{\mathbb{K}}}
\newcommand{\LL}{{\mathbb{L}}}
\newcommand{\NN}{{\mathbb{N}}}
\newcommand{\PP}{{\mathbb{P}}}
\newcommand{\QQ}{{\mathbb{Q}}}
\newcommand{\RR}{{\mathbb{R}}}
\newcommand{\ZZ}{{\mathbb{Z}}}
\newcommand\ONE{\mathbbm{1}}

\newcommand{\bfa}{{\boldsymbol a}}
\newcommand{\bfb}{{\boldsymbol b}}
\newcommand{\bfc}{{\boldsymbol c}}
\newcommand{\bfd}{{\boldsymbol d}}
\newcommand{\bfe}{{\boldsymbol e}}
\newcommand{\bff}{{\boldsymbol f}}
\newcommand{\bfg}{{\boldsymbol g}}
\newcommand{\bfi}{{\boldsymbol i}}
\newcommand{\bfj}{{\boldsymbol j}}
\newcommand{\bfp}{{\boldsymbol p}}
\newcommand{\bfr}{{\boldsymbol r}}
\newcommand{\bfs}{{\boldsymbol s}}
\newcommand{\bft}{{\boldsymbol t}}
\newcommand{\bfu}{{\boldsymbol u}}
\newcommand{\bfv}{{\boldsymbol v}}
\newcommand{\bfw}{{\boldsymbol w}}
\newcommand{\bfx}{{\boldsymbol x}}
\newcommand{\bfy}{{\boldsymbol y}}
\newcommand{\bfz}{{\boldsymbol z}}
\newcommand{\bfA}{{\boldsymbol A}}
\newcommand{\bfF}{{\boldsymbol F}}
\newcommand{\bfB}{{\boldsymbol B}}
\newcommand{\bfD}{{\boldsymbol D}}
\newcommand{\bfG}{{\boldsymbol G}}
\newcommand{\bfI}{{\boldsymbol I}}
\newcommand{\bfM}{{\boldsymbol M}}
\newcommand{\bfP}{{\boldsymbol P}}
\newcommand{\bfzero}{{\boldsymbol{0}}}
\newcommand{\bfone}{{\boldsymbol{1}}}

\newcommand{\Aut}{\operatorname{Aut}}
\newcommand{\Bernoulli}{\operatorname{\textsf{B}}}
\newcommand{\characteristic}{\operatorname{char}}
\newcommand{\Crit}{\operatorname{Crit}}
\newcommand{\CritVal}{\operatorname{CritVal}}
\newcommand{\Curve}{{\Gamma}}
\newcommand{\Disc}{\operatorname{Disc}}
\newcommand{\Div}{\operatorname{Div}}
\newcommand{\Dom}{\operatorname{Dom}}
\newcommand{\End}{\operatorname{End}}
\newcommand{\Fbar}{{\bar{F}}}
\newcommand{\Fix}{\operatorname{Fix}}
\newcommand{\Gal}{\operatorname{Gal}}
\newcommand{\GL}{\operatorname{GL}}
\newcommand{\Hom}{\operatorname{Hom}}
\newcommand{\id}{{\textup{id}}}
\newcommand{\Image}{\operatorname{Image}}
\newcommand{\Isom}{\operatorname{Isom}}
\newcommand{\hhat}{{\hat h}}
\newcommand{\Ker}{{\operatorname{ker}}}
\newcommand{\LCM}{\operatorname{LCM}}
\newcommand{\limstar}{\lim\nolimits^*}
\newcommand{\limstarn}{\lim_{\hidewidth n\to\infty\hidewidth}{\!}^*{\,}}
\newcommand{\Mat}{\operatorname{Mat}}
\newcommand{\maxplus}{\operatornamewithlimits{\textup{max}^{\scriptscriptstyle+}}}
\newcommand{\MOD}[1]{~(\textup{mod}~#1)}
\newcommand{\Mor}{\operatorname{Mor}}
\newcommand{\Moduli}{\mathcal{M}}
\newcommand{\Norm}{{\operatorname{\mathsf{N}}}}
\newcommand{\notdivide}{\nmid}
\newcommand{\normalsubgroup}{\triangleleft}
\newcommand{\NS}{\operatorname{NS}}
\newcommand{\onto}{\twoheadrightarrow}
\newcommand{\ord}{\operatorname{ord}}
\newcommand{\Orbit}{\mathcal{O}}
\newcommand{\Orb}{\operatorname{Orb}}
\newcommand{\Per}{\operatorname{Per}}
\newcommand{\PrePer}{\operatorname{PrePer}}
\newcommand{\PGL}{\operatorname{PGL}}
\newcommand{\Pic}{\operatorname{Pic}}
\newcommand{\Prob}{\operatorname{Prob}}
\newcommand{\Proj}{\operatorname{Proj}}
\newcommand{\Qbar}{{\overline{\QQ}}}
\newcommand{\rank}{\operatorname{rank}}
\newcommand{\Rat}{\operatorname{Rat}}
\newcommand{\red}{{\textup{red}}}
\newcommand{\Resultant}{\operatorname{Res}}
\renewcommand{\setminus}{\smallsetminus}
\newcommand{\sgn}{\operatorname{sgn}} 
\newcommand{\SL}{\operatorname{SL}}
\newcommand{\Span}{\operatorname{Span}}
\newcommand{\Spec}{\operatorname{Spec}}
\renewcommand{\ss}{\textup{ss}}
\newcommand{\stab}{\textup{stab}}
\newcommand{\Support}{\operatorname{Supp}}
\newcommand{\tors}{{\textup{tors}}}
\newcommand{\tr}{{\textup{tr}}} 
\newcommand{\Trace}{\operatorname{Trace}}
\newcommand{\trianglebin}{\mathbin{\triangle}} 
\newcommand{\UHP}{{\mathfrak{h}}}    
\newcommand{\<}{\langle}
\renewcommand{\>}{\rangle}

\newcommand{\pmodintext}[1]{~\textup{(mod}~#1\textup{)}}
\newcommand{\ds}{\displaystyle}
\newcommand{\longhookrightarrow}{\lhook\joinrel\longrightarrow}
\newcommand{\longonto}{\relbar\joinrel\twoheadrightarrow}
\newcommand{\SmallMatrix}[1]{%
  \left(\begin{smallmatrix} #1 \end{smallmatrix}\right)}

\newcommand{\TABT}[1]{\begin{tabular}[t]{@{}l@{}}#1\end{tabular}}
\newcommand{\TAB}[1]{\begin{tabular}{@{}l@{}}#1\end{tabular}}
\newcommand{\TABC}[1]{\begin{tabular}{@{}c@{}}#1\end{tabular}}

\begin{abstract} 
Let~$\mathbb{F}$ be the function field of a curve over an
algebraically closed field with
$\operatorname{char}(\mathbb{F})\ne2,3$, and let~$E/\mathbb{F}$ be a
non-isotrivial elliptic curve. Then for all finite 
extensions~$\mathbb{K}/\mathbb{F}$ and all non-torsion points
$P\in{E(\mathbb{K})}$, the $\mathbb{F}$-normalized canonical height
of~$P$ is bounded below by
\[
\hat{h}_E(P) \ge \frac{1}{10500\cdot h_{\mathbb{F}}(j_E)^{2}\cdot [\mathbb{K}:\mathbb{F}]^{2}}.
\]
\end{abstract}

\maketitle

\tableofcontents

\section{Introduction}
\label{sec:introduction}

The classical Lehmer conjecture for number fields says that there is
an absolute constant~$\Cl{KoverQ}>0$ such that for all number
fields~$K/\QQ$ of degree~\text{$D=[K:\QQ]$} and all~$\a\in{K^*}$ that
are not roots of unity, the absolute logarithmic Weil height of~$\a$
satisfies
\[
h(\a) \ge \frac{\Cr{KoverQ}}{D}.
\]
The best known result~\cite{Dobrowolski} gives a lower bound with an
additional factor of~$(\log\log D/\log D)^3$, while the best known
results for the analogous conjecture for the canonical height on an
elliptic curve~$E$ and non-torsion point~$P\in{E(K)}$ are as follows:
\begin{align*}
  \hhat_E(P) &\ge \Cl{EoverK1}(E)\cdot D^{-3} (\log D)^{-2}
  &&\text{in general~\cite{Masser89},} \\
  \hhat_E(P) &\ge \Cl{EoverK2}(E)\cdot D^{-2} (\log D)^{-2}
  &&\text{if $j_E$ is non-integral~\cite{HS-STNP1990},} \\
  \hhat_E(P) &\ge \Cl{EoverK3}(E)\cdot D^{-1} (\log\log D/\log D)^{3}
  &&\text{if $E$ has CM~\cite{Laurent}.} 
\end{align*}
\par
In this paper we consider the function field analogue of Lehmer's
conjecture for elliptic curves. We set the following notation, which
will be used throughout this article.
\begin{equation}
  \label{eqn:definitions}
  \framebox{\parbox{.9\hsize}{%
    \begin{tabular}{c@{\quad}l}
    $k$ & 
      an algebraically closed field with $\characteristic(k)\ne2,3$. \\ 
    $\FF$ & 
      the function field of a smooth projective curve\\
      &defined over $k$. \\
    $h_\FF$ & 
      the absolute Weil height $h_\FF:\PP^n(\FFbar)\to[0,\infty)$. \\
    $E/\FF$ & 
      an elliptic curve defined over $\FF$ with $j_E\notin k$. \\ 
    $\hhat_{E/\FF}$ & 
      the absolute logarithmic canonical height on $E(\FFbar)$. \\
    \end{tabular}
}}
\end{equation}

\begin{conjecture}[Lehmer Conjecture for an Elliptic Curve over a Function Field]
\label{conjecture:lehmerecoverff}
With notation as in~\eqref{eqn:definitions},
there is a constant~$\Cl{lehmer}(E/\FF)>0$ such that
for all finite extensions~$\KK/\FF$ of degree~\text{$D=[\KK:\FF]\ge2$},
we have
\begin{equation}
  \label{eqn:lehmerecoverffest}
  \hhat_{E/\FF}(P) \ge \frac{\Cr{lehmer}(E/\FF)}{D}
  \quad\text{for all non-torsion $P\in E(\KK)$.}
\end{equation}
\end{conjecture}

Our main theorem gives a lower bound of the
form~$\Cl{lehmerx}(E/\FF)/D^2$, with an explicit value for the
constant~$\Cr{lehmerx}(E/\FF)$.
\begin{theorem}
\label{theorem-1}
With notation as in~\eqref{eqn:definitions},
let~$\KK/\FF$ be a finite extension of degree~\text{$D=[\KK:\FF]$}. Then
\[
\hhat_{E/\FF}(P) \ge \frac{1}{10500 \cdot h_\FF(j_E)^2\cdot D^2}
\quad\text{for all non-torsion $P\in E(\KK)$.}
\]
The constant can be improved for large values of~$D$. For example, 
\begin{multline*}
h_\FF(j_E)^{3/2} \cdot D \ge 50000 \\
\Longrightarrow\;
\hhat_{E/\FF}(P) \ge \frac{1}{4108 \cdot h_\FF(j_E)^2\cdot D^2}
\quad\text{for all non-torsion $P\in E(\KK)$.}
\end{multline*}
\end{theorem}

As noted in numerous papers, estimates as in Theorem~\ref{theorem-1}
are an immediate consequence of the following sort of counting result:

\begin{theorem}
\label{theorem-2}
With notation as in~\eqref{eqn:definitions},
let~$\KK/\FF$ be a finite extension of
degree~\text{$D=[\KK:\FF]$}. Then for all
\[
0<\d<1\quad\text{and}\quad 0<\e<1,
\]
we have
\begin{multline*}
  \#\left\{ P\in E(\KK) : \hhat_{E/\FF}(P) \le 
  \frac{1-\d}{2} \left( \frac{1}{2^{5}\cdot(1+\e)\cdot D} \right)^{2/3}
  \right\} \\
  \le
  \max\left\{\frac{12}{\e},\,
  \frac{h_\FF(j_E)}{\d} \left(2^5\cdot(1+\e)\cdot D\right)^{2/3}\right\}.
\end{multline*}
\end{theorem}

The proof of Theorem~\ref{theorem-2} uses Fourier averaging ideas
as in~\cite{HS-STNP1990}, with some care taken to optimize the argument
and to keep track of how the constants depend on the curve~$E$.

\begin{remark}
\label{remark:philosophy}
We remark that the classical Lehmer conjecture for~$\KK^*$ is trivial
in the function field setting. Philosophically, we view the difficulty
in Lehmer's conjecture as coming from places of persistant bad reduction. The
group scheme~$\GG_m$ has everywhere good reduction over a function
field, whence the triviality of a Lehmer estimate.  However, we also
philosophically view~$\GG_m$, and indeed all abelian group schemes, as
having persistant bad reduction at all archimedean places, which explains why
Lehmer estimates are non-trivial for~$\GG_m$ over number fields. As we
shall see, the challenge in proving a Lehmer-type estimate for an
elliptic curve over a function field comes almost entirely from the
places of multiplicative reduction, i.e., from the places of stable
bad reduction.
\end{remark}

\begin{remark}
Conjecture~\ref{conjecture:lehmerecoverff} and
Theorems~\ref{theorem-1} and~\ref{theorem-2} all include the
assumption that~$E/\FF$ is not isotrivial, i.e., they assume
that~$j_E\notin{k}$. This is necessary, since the conjecture is false
in general if~$j_E\in{k}$; but if we restrict to extensions~$\KK/\FF$
such that~$E/\KK$ is not split, then Conjecture~\ref{conjecture:lehmerecoverff}
is true and relatively easy to prove. See Section~\ref{section:finalremarks0}
for details.
\end{remark}

We briefly describe the contents of this paper.
Section~\ref{sec:notation} sets notation, and in particular explains
how absolute heights are normalized over the base function
field~$\FF$.  Section~\ref{sec:localcanonicalheights} recalls the
decomposition of the canonical height into local heights and gives
lower bounds for local heights.  Section~\ref{sec:reprise} gives two
useful elliptic curve results from~\cite{HS-STNP1990}, together with a
strengthened version of a numerical inequality appearing
in~\cite{HS-STNP1990} and~\cite[Lemma~9.4]{LS-TransAMS2024}.  The
proof of the latter is deferred to Appendix~\ref{section:inquality}.
Section~\ref{sec:proofsmainthm} contains the proofs of our main
results.  In Section~\ref{section:finalremarks0} we illustrate the
philosophy described in Remark~\ref{remark:philosophy} by
analyzing  isotrivial~$E$.
Finally, in Appendices~\ref{section:localhtImstar}
and~\ref{section:localhtImstarintersection}, we give two derivations
of formulas for the local height in the case of additive, potentially
multiplicative, reduction, i.e., reduction of type~$I_M^*$, the first
via explicit Weierstrass equations and Tate's algorithm, the second
via intersection theory. We include this material partly because it is
not readily available in the literature, and partly to illustrate and
contrast the two approaches.

\section{Notation}
\label{sec:notation}

In addition to the notation set in~\eqref{eqn:definitions},
we set the following notation.
The field~$\FF$ is the function field of a smooth projective curve,
and each point on that curve gives rise to a normalized valuation
\begin{equation}
  \label{eqn:ordvFstarontoZ}
  \ord_v : \FF^* \longonto \ZZ.
\end{equation}
We let
\[
M_\FF = \text{the set of normalized valuations as described by \eqref{eqn:ordvFstarontoZ}.}
\]
IN the terminology of Lang~\cite{MR715605}, the
set~$M_\FF$ is a proper set of absoltue values on~$\FF$ satisfying the
product formula with all multiplicities equal to~$1$, and thus there is a
general theory of Weil and N{\'e}ron--Tate height functions
over~$\FF$.  (See also Remark~\ref{remark:properabsvals} at the end of
this section.)
\par
For a finite extension~$\KK/\FF$, we define~$M_\KK$ to be the set
of valuations on~$\KK$ that extend the valuations in~$M_\FF$, i.e., an
element~$w\in{M_\KK}$ is a valuation
\[
w : \KK^*\longrightarrow\QQ\quad\text{satisfying}\quad w|_{\FF} \in M_\FF.
\]
For~$w\in{M_\KK}$ and~$v\in{M_\FF}$, we write~\text{$w\mid{v}$} to
indicate that~$w|_\FF=v$.
\par
We denote the local index of~$w$ over~$w|_\FF$ by\footnote{We note that our
assumptions imply that for the residue field of~$w\in{M_\KK}$ is
isomorphic to the algebraically closed field~$k$, so writing~$\KK_w^s$
for the maximimal separable subextension of~$\KK/\FF$, the local
degree~$e(w)$ is equal to the product of the ramification degree
of~$\KK_w^s/\FF$ and the inseperable degree of the
extension~$\KK/\FF$.}
\[
e(w) = [\KK_w:\FF_w],
\]
and then we have
\begin{equation}
  \label{eqn:sumeifi}
  \sum_{\substack{w\in M_\KK\\ w\mid v\\}} e(w) = [\KK:\FF]
  \quad\text{for all $v\in M_\FF$.}
\end{equation}
We note that normalized valuation associated to~$w$
is related to~$w$ via
\[
\ord_w:\KK^*\longonto\ZZ,\quad \ord_w = e(w)w.
\]

\begin{definition}
Let~$\KK/\FF$ be a finite extension, and 
let~$P=[\a_0,\ldots,\a_n]\in\PP^n(\KK)$.  The \emph{$\FF$-normalized
height of~$P$} is\footnote{By convention we set $w(0^{-1})=0$.}
\begin{equation}
  \label{eqn:defweilht}
  h_\FF(P) := \frac{1}{[\KK:\FF]} \sum_{w\in M_\KK} e(w) \max_{0\le i\le n} w(\a_i^{-1}).
\end{equation}
The height is invariant under field extensions, so it extends to a
well-defined function
\[
h_\FF : \PP^n(\FFbar) \longrightarrow [0,\infty).
\]
\end{definition}

\begin{definition}
Let~$x$ be the~$x$-coordinate on a Weierstrass equation for~$E/\FF$,
i.e., let~\text{$x:E\to\PP^1$} be a non-constant global section
of the sheaf~$\Ocal_{E/\FF}\bigl(2(O)\bigr)$.  Then the \emph{canonical height}
on~$E(\FFbar)$ is defined by the usual limit
\[
\hhat_{E/\FF} : E(\FFbar)\longrightarrow[0,\infty),\quad
\hhat_{E/\FF}(P) := \lim_{n\to\infty} 4^{-n}h_\FF\bigl(x(2^nP)\bigr).
\]
\end{definition}

\begin{remark}
\label{remark:properabsvals}
For the convenience of the reader, we recall some terminology
from~\cite{MR715605}.  A non-archmedean absolute value~$v$ on a
field~$\FF$ is \emph{proper} if it is non-trivial and
satisfies~$[\KK:\FF]=\sum_{w\mid{v}}[\KK_w:\FF_v]$ for every finite
extension~$\KK/\FF$.  A set of non-archmedean absolute values~$M_\FF$
on~$\FF$ is \emph{proper} if (1)~every~$v\in{M_\FF}$ is proper;
(2)~distinct elements of~$M_\FF$ are independent; and (3)~for
all~$\a\in\FF^*$, we have~$|\a|_v=1$ for all but finitely
many~$v\in{M_\FF}$. The set~$M_\FF$ satisfies the \emph{product
formula with multiplicities~$\mu_v$}
if~$\prod_{v\in{M_K}}|\a|_v^{\mu_v}=1$ for all~$\a\in\FF^*$.  Then for
finite~$\KK/\FF$, the set of absolute values~$M_\KK$ on~$\KK$
extending the absolute values on~$\FF$ is proper and satisfies the
product formula with multiplicities~\text{$[\KK_w:\FF_v]\mu_v$}
for~$w\mid{v}$.
\end{remark}

\section{A Parallelogram Law Canonical Height Inequality}
\label{section:parallelogramlawht}
We use the parallelogram law for the canonical height~\cite[Theorem~VIII.9.3]{AEC} to
derive an inequality that will be applied later.
\begin{lemma}
\label{lemma:hhatineq}
Let~$P_1,\ldots,P_N\in{E(\FFbar)}$. Then
\begin{equation}
  \label{eqn:maxhhatPisumPjPk}
  \max_{0\le i\le N} \hhat_{E/\FF}(P_i)
  \ge \frac{1}{2(N+1)^2}\sum_{\substack{j,k=0\\ j\ne k\\}}^N  \hhat_{E/\FF}(P_j-P_k).
\end{equation}
Further, the inequality~\eqref{eqn:maxhhatPisumPjPk} is an equality if
and only if
\[
\hhat_{E/\FF}(P_0)=\cdots=\hhat_{E/\FF}(P_N)\quad\text{and}\quad
P_0+\cdots+P_N\in{E(\FFbar)_\tors}.
\]
\end{lemma}
\begin{proof}
We denote the bilinear form associated to the canonical height by
\[
\<P,Q\>_{E/\FF} := \frac12 \Bigl( \hhat_{E/\FF}(P+Q)-\hhat_{E/\FF}(P)-\hhat_{E/\FF}(Q) \Bigr).
\]
We compute
\begin{align}
  \smash[b]{ \sum_{\substack{j,k=0\\ j\ne k\\}}^N  \hhat_{E/\FF}(P_j-P_k) }
  &=  \sum_{j,k=0}^N  \hhat_{E/\FF}(P_j-P_k)  \quad\text{since $\hhat_{E/\FF}(O)=0$,} \notag\\*
  &= \sum_{j,k=0}^N
  \Bigl(\hhat_{E/\FF}(P_j)-2\bigl\<P_j,P_k\bigr\>_{E/\FF}+\hhat_{E/\FF}(P_k)\Bigr) \notag\\*
  &= \sum_{j,k=0}^N  \Bigl(\hhat_{E/\FF}(P_j)+\hhat_{E/\FF}(P_k)\Bigr)
  - 2 \left\<\sum_{j=0}^N P_j,\sum_{k=0}^N P_k\right>_{\!\!E/\FF} \notag\\
  &= 2(N+1) \sum_{i=0}^N \hhat_{E/\FF}(P_i)
  - 2\hhat_{E/\FF}\biggl(\sum_{i=0}^N P_i\biggr) \notag\\*
  \label{eqn:maxhhatPisumPjPk1}
  &\le 2(N+1) \sum_{i=0}^N \hhat_{E/\FF}(P_i) \quad\text{since $\hhat_{E/\FF}\ge0$,}\\*
  \label{eqn:maxhhatPisumPjPk2}
  &\le 2(N+1)^2 \max_{0\le i\le N} \hhat_{E/\FF}(P_i).
\end{align}
This proves~\eqref{eqn:maxhhatPisumPjPk}. We conclude by noting that
the inequality~\eqref{eqn:maxhhatPisumPjPk1} is an equality if and
only if~$\hhat_{E/\FF}(\sum{P_i})=0$, so if and only
if~$\sum{P_i}\in{E_\tors}$, and the
inequality~\eqref{eqn:maxhhatPisumPjPk2} is an equality if and only
if the~$\hhat_{E/\FF}(P_i)$ are all equal.
\end{proof}
\section{Local Canonical Heights}
\label{sec:localcanonicalheights}

We recall the principal properties of local canonical heights
on elliptic curves.

\begin{theorem}
\label{theorem:localhtexist}
\textup{(N{\'e}ron--Tate)}
With notation as in~\eqref{eqn:definitions},
we fix  a Weierstrass equation
\begin{equation}
  \label{eqn:WE1}
  E : y^2+a_1xy+a_3y=x^3+a_2x^2+a_4x+a_6
\end{equation}
for~$E/\FF$,  we let~$\Delta$ be the
associated discriminant~\cite[III~\S1]{AEC},
and we let~$\KK/\FF$ be a finite  extension. 
\begin{parts}
\Part{(a)}
For each~$w\in{M_\KK}$, there is a unique local canonical
height\footnote{The elliptic curve and base field are implicit in the
notation, but should we need to specify them, we will write
$\lhat_{E/\FF,w}$.}
\[
\lhat_w : E(\KK_w)\setminus0 \longrightarrow\RR
\]
having the following three properties\textup{:}
\begin{parts}
\Part{(i)}
$\lhat_w$ is continuous and bounded on the complement of any~$w$-adic
neighborhood of~$0$.
\Part{(ii)}
The limit
\[
  \lim_{\substack{P\to0\\\textup{$w$-adic}\\}}
  \left[ \lhat_w(P) + \frac{1}{2} w\bigl(x(P)\bigr)\right]
\]
exists.
\Part{(iii)}
For all~$P\in{E(\KK_w)}$ with~$2P\ne0$, we have
\[
\lhat_w(2P) =
4\lhat_w(P) + \frac{1}{4} w\left( \frac{\bigl(2y(P)+a_1x(P)+a_3\bigr)^4}{\D} \right).
\]
\end{parts}
\Part{(b)}
With local heights normalized as in~\textup{(a)}, we have
\begin{equation}
  \label{eqn:hhatPsumlhatwP}
  \hhat(P) = \frac{1}{[\KK:\FF]} \sum_{w\in M_\KK} e(w)\lhat_w(P)
  \quad\text{for all $P\in{E(\KK)}\setminus0$.}
\end{equation}
\Part{(c)}
Let~$\LL/\KK$ be a finite extension, let~$w\in{M_\KK}$, and
let~$u\in{M_\LL}$ with~\text{$u\mid{w}$}. Then
\[
\lhat_u(P) = \lhat_w(P) \quad\text{for all $P\in E(\KK_w)\subseteq E(\LL_u)$.}
\]
\end{parts}
\end{theorem}
\begin{proof}
(a)\enspace
We refer the reader to~\cite[Theorem~VI.1.1]{ATAEC} for a proof of~(a).
We remark that using the normalization provided by~(iii), the limit
in~(ii) can be computed explicitly and shown to
equal~$\frac{1}{12}w(\D)$; see Remark~\ref{remark:limlhatwvalue} at
the end of this section for a proof.  We also note that the
quantities~$x^6/\D$ and~$(2y+a_1x+a_3)^4/\Delta$ are well-defined
functions on~$E$ that do not depend on the choice of Weierstrass
equation, since any other Weierstrass equation for~$E$ is related
via
\[
x=u^2x',\quad y=u^3y',\quad a_1=ua_1',\quad a_3=u^3a_3',\quad \Delta=u^{12}\Delta'
\]
for some non-zero~$u$; see~\cite[III~\S1]{AEC}.  Thus the existence of
the limit in~(ii) and the equality in~(iii) are independent of the
choice of Weierstrass equation.
\par\noindent(b)\enspace
This is proved for number fields
in~\cite[Theorem~VI.2.1]{ATAEC}. However, the proof does not need the
number field assumption. It relies only on the following facts: (1)~$\KK/\FF$ is
an extension of fields with absolute values satisfying the product
formula; (2)~$\lhat_w$ has the properties described in~(a); (3)~the
formula for~$\lhat_w(P)$ at points
in~$E_0(\KK_w)$, which is proven for general complete fields in~\cite[Theorem~VI.4.1]{ATAEC}.
\par\noindent(c)\enspace
We have normalized the valuations so that the restriction of~$u$
to~$\KK_w$ is equal to~$w$. Hence the restriction of~$\lhat_u$
to~$E(\KK_w)$ satisfies~(i)--(iii), so by the uniqueness of~$\lhat_w$,
we have~$\lhat_u\big|_{E(\KK_w)}=\lhat_w$.
\end{proof}

We next give two definitions and then recall/prove some explicit
formulas and estimates for local canonical heights on elliptic curves.

\begin{definition}
\label{definition:Jsubw}
For~$w\in{M_\KK}$ and~$v\in{M_\FF}$ with~\text{$w\mid{v}$}, we
define\footnote{We remark that in~\cite{HS-STNP1990}, our
quantity~$J_w$ was denoted by~$\nu_w$. We have switched notation,
because~$\nu$ and~$v$ look too much alike.}
\begin{align*}
  J_w &:= \max\bigl\{ 0, \ord_w(j_E^{-1}) \bigr\} \\
  &= \max\bigl\{ 0, e(w)\ord_v(j_E^{-1}) \bigr\} \quad\text{since $j_E\in \FF$,}\\
  &= e(w) \max\bigl\{ 0,\ord_v(j_E^{-1}) \bigr\} \\
  &= e(w)\cdot \max\bigl\{ 0,v(j_E^{-1}) \bigr\} \quad\text{since $v=\ord_v$ for $v\in M_\FF$.}
\end{align*}
\end{definition}

\begin{definition}
The \emph{periodic second Bernoulli polynomial} is the function
\[
\Bernoulli_2 : \RR \longrightarrow \RR/\ZZ \longrightarrow \RR,
\]
defined by
\[
\Bernoulli_2(t) = t^2 - t + \frac16\quad\text{for $0\le t\le 1$,}
\]
and extended periodically by requiring that
\[
\Bernoulli_2(t+n) = \Bernoulli_2(t) \quad\text{for all $t\in\RR$ and all $n\in\ZZ$.}
\]
\end{definition}

\begin{lemma}
\label{lemma:potentialgoodreduction}
With notation as in~\eqref{eqn:definitions},
let \text{$v\in{M_\FF}$},
let~$\KK/\FF$ be a finite  extension, 
let~\text{$w\in{M_\KK}$} with~\text{$w\mid{v}$},
and let
\[
w(\Dcal_{E/\KK_w}) = \text{valuation of the minimal discriminant of $E$ at~$w$.}
\]
\vspace*{-10pt}
\begin{parts}
\Part{(a)}
For~$w\in{M_\KK}$, let~$E_\nu(\KK_w)$ for~$\nu\ge0$ be the formal group filtration
of~$E(\KK_w)$, so in particular~$E(\KK_w)/E_0(\KK_w)$ is the group of components of
the N{\'e}ron model of~$E$ at~$w$. Then for~$\nu\ge0$ we have
\begin{equation}
  \label{eqn:lhatwa1}
  \lhat_w(P) = \nu + \frac{1}{12}w(\Dcal_{E/\KK_w})
  \quad\text{for all $P\in E_\nu(\KK_w)\setminus E_{\nu+1}(\KK_w)$.}
\end{equation}
In particular, if~$E$ has additive or good reduction at~$w$, then
\begin{equation}
  \label{eqn:lhatwa2}
  \lhat_w(Q) \ge \frac{1}{12}w(\Dcal_{E/\KK_w}) \quad\text{for all $Q\in12E(\KK_w)$.}
\end{equation}
\Part{(b)}
If \text{$v(j_E)\ge0$}, or equivalently if~$E$ has potential good reduction at~$v$,
then
\[
\lhat_w(P) \ge 0 \quad\text{for all $P\in E(\KK_w)\setminus0$.}
\]
\Part{(c)}
If
\[
w(\Dcal_{E/\KK_w}) = v(j_E^{-1}) > 0, 
\]
or equivalently if~$E$ has multiplicative reduction at~$w$, then
\[
\lhat_w(P)\ge \frac12 \Bernoulli_2\bigl(\f_w(P)\bigr) v(j_E^{-1}),
\]
where~$\f_w$ is a homomorphism
\[
\f_w : E(\KK_w) \longrightarrow \tfrac{1}{J_w}\ZZ/\ZZ \subset \QQ/\ZZ.
\]
\Part{(d)}
If
\[
w(\Dcal_{E/\KK_w}) - 6 = v(j_E^{-1}) > 0,
\]
or equivalently if~$E$ has additive, potential multiplicative reduction, at~$w$, then
\[
\lhat_w(P)\ge \frac12 \Bernoulli_2\bigl(\f_w(P)\bigr) v(j_E^{-1}),
\]
where~$\f_w$ is a homomorphism
\[
\f_w : E(\KK_w) \longrightarrow \tfrac12\ZZ/\ZZ \subset \QQ/\ZZ.
\]
\end{parts}
\end{lemma}
\begin{proof}
(a)\enspace
Let~$x$ be the~$x$-coordinate on a minimal Weierstrass equation for~$E$ at~$w$.
Then~\cite[Theorem~VI.4.1]{ATAEC} says that
\[
\lhat_w(P) = \frac12\max\bigl\{ w\bigl(x(P)^{-1}\bigr), 0\bigr\} + \frac{1}{12}w(\Dcal_{E/\KK})
\quad\text{for all $P\in E_0(\KK_w)$.}
\]
The filtration of~$E(\KK_w)$ is given by~\cite[Chapter~VII]{AEC}
\[
P\in E_\nu(\KK_w) \quad\Longleftrightarrow\quad
\frac12\max\bigl\{ w\bigl(x(P)^{-1}\bigr), 0\bigr\} \ge \nu,
\]
which gives~\eqref{eqn:lhatwa1}.
\par
We next observe that the list of Kodaira--N{\'e}ron fibers for
additive reduction implies that\footnote{More precisely,
\par\noindent\hbox to\textwidth{\hfill
$\begin{array}{|c||c|c|c|c|c|} \hline
  \text{Type} & \text{II \&\ II$^*$} & \text{III \&\ III$^*$}
  & \text{IV \&\ IV$^*$} & \text{I$_{2M}^*$} & \text{I$_{2M+1}^*$} \\ \hline
  E(\KK_w)/E_0(\KK_w) & 1 & \ZZ/2\ZZ & \ZZ/3\ZZ & \ZZ/2\ZZ\times\ZZ/2\ZZ
  & \ZZ/4\ZZ \\ \hline
\end{array}$\hfill}
}
\[
12P\in E_0(\KK_w) \quad\text{for all $P\in E(\KK_w)$;}
\]
see~\cite[IV~\S8, Figure~4.4]{ATAEC} or~\cite{Antwerp4Tate}.
Hence we can apply~\eqref{eqn:lhatwa1} to points in~$12E(\KK)$, which
gives~\eqref{eqn:lhatwa2}.
\par\noindent(b)\enspace
The assumption that~$v(j_E)\ge0$ implies that~$E$ has potential good
reduction at~$v$~\cite[Proposition~VII.5.5]{AEC}, so we can find a
finite (separable) extension~$\LL/\KK$ and a valuation~$w'\in{M_\LL}$ 
satisfying~\text{$w'\mid{w}\mid{v}$} such that~$E$ has good reduction
at~$w'$.\footnote{The separability follows from the fact that~$\LL$ can
be generated over~$\KK$ by adjoining square and cube roots, and we
have assumed that~$\characteristic(\FF)\ne2,3$.}
It follows from~\cite[Theorem~VI.4.1]{ATAEC} (see
also~\cite[Remark~VI.4.1.1]{ATAEC}) that if we take a minimal
Weierstrass equation for~$E/\LL_u$, then
\[
\lhat_{w'}(P) = \frac12 \max\Bigl\{ w'\bigl(x(P)^{-1}\bigr),\,0\Bigr\}.
\quad\text{for all $P\in E(\LL_{w'})\setminus0$.}
\]
In particular, we see that
\[
\lhat_{w'}(P) \ge 0 \quad\text{for all $P\in E(\LL_u)\setminus0$.}
\]
We also know that the local height is invariant under field extension
(Theorem~\ref{theorem:localhtexist}(c)),
i.e.,~$\lhat_{w'}|_{E(\KK_w)}=\lhat_w$, which allows us to conclude that
\[
\lhat_w(P) \ge 0 \quad\text{for all $P\in E(\KK_w)\setminus0$.}
\]
\par\noindent(c)\enspace
This follows from Tate's explicit formula for the local height in the case of
multiplicative reduction. See~\cite[Theorem~VI.4.2]{ATAEC} for a proof
for~$p$-adic fields, i.e., for finite extensions of~$\QQ_p$. However,
we note that the proof
in~\cite{ATAEC} requires only Tate's parameterization of
multiplicative reduction elliptic curves, and such parameterizations
exist for all complete non-archimedean
local fields. See for example~\cite{RoquetteTateCurve,TateTateCurve}.
\par\noindent(d)\enspace
Our assumptions say that~$E$ has reduction type~$I_M^*$ at~$w$ for some
integer~\text{$M\ge1$},
where\footnote{We note that~$M\ne0$ since~$I_0^*$ is a potential good reduction and
we've assumed potential multiplicative reduction. And the equalities in~\eqref{eqn:IMstarstuff}
are valid since~$\characteristic(\KK)\ne2$.}
\begin{equation}
  \label{eqn:IMstarstuff}
  M = w(\Dcal_{E/\KK_w}) - 6 = w(j_E^{-1}) = v(j_E^{-1}).
\end{equation}
\par
The special fiber of the N{\'e}ron model for~$I_M^*$ reduction
consists of four multiplicity-$1$ components; see
Figure~\ref{figure:IMstar} in Appendix~\ref{section:localhtImstar}, where we have labeled the
components~$\bfzero$,~$\bfalpha$,~$\bfbeta$,
and~$\bfbeta'$. We remark that~$E_0(\KK_w)$ is the set
of points that intersect the~$\bfzero$-component.  
Formulas for~$\lhat_w(P)$ may be found in various sources, but for
completeness we give two proofs, one in
Appendix~\ref{section:localhtImstar} using explicit Weierstrass equations,
and one in Appendix~\ref{section:localhtImstarintersection}
using intersection theory. These formulas imply that
\begin{align*}
  \lhat_w(P) &\ge 0
  && \text{if $P$ hits the $\bfzero$-component,} \\
  \lhat_w(P) &= \frac{1}{12}v(j_E^{-1})
  && \text{if $P$ hits the $\bfalpha$-component,} \\
  \lhat_w(P) &= -\frac{1}{24}v(j_E^{-1})
  && \text{if $P$ hits the $\bfbeta$ or $\bfbeta'$-component.} 
\end{align*}
Finally, we note that the group of components~$E(\KK_w)/E_0(\KK_w)$ is
either~$(\ZZ/2\ZZ)^2$ of~$(\ZZ/4\ZZ)$ depending on whether~$M$ is even or odd,
with the $\bfbeta$ and $\bfbeta'$-components being the elements of order~$4$ in
the latter case. These formulas give the desired result. 
\end{proof}

\begin{remark}
\label{remark:limlhatwvalue}
We conclude this section by evaluating the limit
\[
L_w :=
\lim_{\substack{P\to0\\\textup{$w$-adic}\\}}
\left[ \lhat_w(P) + \frac{1}{2} w\bigl(x(P)\bigr)\right]
\]
that appears in Theorem~\ref{theorem:localhtexist}(a-ii).  We compute
\begin{align*}
  L_w
  &=
  \lim_{\substack{P\to0\\\textup{$w$-adic}\\}}
  \left\{-\frac13\left[ \lhat_w(2P) + \frac{1}{2} w\bigl(x(2P)\bigr)\right]
  +\frac43\left[ \lhat_w(P) + \frac{1}{2} w\bigl(x(P)\bigr)\right]\right\} \\
  &\omit\hfill\text{since~$P\to0$ implies that~$2P\to0$,
  so the limit is $-\frac13L_w+\frac43L_w$} \\
  &= 
  \lim_{\substack{P\to0\\\textup{$w$-adic}\\}}
  \left\{-\frac13\left[
    4\lhat_w(P) + \frac{1}{4} w\left( \frac{\bigl(2y(P)+a_1x(P)+a_3\bigr)^4}{\D} \right)
    + \frac{1}{2} w\bigl(x(2P)\bigr)\right] \right.\\
  &\omit\hfill$\displaystyle{}+
  \left. \frac43\left[ \lhat_w(P) + \frac{1}{2} w\bigl(x(P)\bigr)\right]\right\}$
  \quad\text{from Theorem~\ref{theorem:localhtexist}(a-iii),}\\
  &=
  \lim_{\substack{P\to0\\\textup{$w$-adic}\\}}
  \frac{1}{12} w\left( \frac{\D x(P)^8}{\bigl(2y(P)+a_1x(P)+a_3\bigr)^4x(2P)^2} \right) \\
  &=
  \lim_{\substack{P\to0\\\textup{$w$-adic}\\}}
  \frac{1}{12} w\left( \frac{\D x(P)^8}{\bigl( x(P)^4-b_4x(P)^2-2b_6x(P)-b_8 \bigr)^2} \right) \\*
  &\omit\hfill\text{duplication formula \cite[Algorithm~III.2.3(d)]{AEC},} \\[-2\jot]
  &= \frac{1}{12} w(\D).
\end{align*}
\end{remark}

\section{Reprise of Material from Previous Papers}
\label{sec:reprise}

\begin{lemma}
\label{lemma:stnp-prop1.2}
\textup{(\cite[Proposition~1.2]{HS-STNP1990})}
With notation as in~\eqref{eqn:definitions},
let \text{$v\in{M_\FF}$} with \text{$v(j_E^{-1})>0$},
let~$\KK/\FF$ be a finite  extension,
let~\text{$w\in{M_\KK}$} with~\text{$w\mid{v}$}, and
let
\[
P_0,\ldots,P_N\in{E}(\KK_w)
\]
be distinct points. Then
\[
\sum_{\substack{j,k=0\\j\ne k\\}}^N \lhat_w(P_j-P_k)
\ge
\frac{1}{12}\left(\frac{N+1}{J_w}\right)^2 v(j_E^{-1}) - \frac{N+1}{12}v(j_E^{-1}),
\]
where~$J_w=e(w)v(j_E^{-1})$ is as in Definition~$\ref{definition:Jsubw}$.
\end{lemma}
\begin{proof}
This is proved in~\cite[Proposition~1.2]{HS-STNP1990} in the case
that~$E$ has multiplicative reduction at~$w$, using the lower bound
for~$\lhat_w$ in Lemma~\ref{lemma:potentialgoodreduction}(c) in terms
of~$\Bernoulli_2$ and the homomorphism
\[
\f_w : E(\KK_w) \longrightarrow J_w^{-1}\ZZ/\ZZ \subset \QQ/\ZZ.
\]
\par
In the case that~$E$ has additive, potential multiplicative,
reduction, we can use the analogous lower bound for~$\lhat_w$ in
Lemma~\ref{lemma:potentialgoodreduction}(d) and repeat the proof
in~\cite[Proposition~1.2]{HS-STNP1990} to obtain the stronger lower bound
\begin{align*}
\sum_{\substack{j,k=0\\j\ne k\\}}^N \lhat_w(P_j-P_k)
&\ge\frac{1}{12}\left(\frac{N+1}{2}\right)^2 v(j_E^{-1}) - \frac{N+1}{12}v(j_E^{-1}) \\*
&= \frac{(N+1)(N-3)}{48}v(j_E^{-1}).
\qedhere
\end{align*}
\end{proof}

\begin{lemma}
\label{lemma:stnp-prop1.3}
\textup{(\cite[Proposition~1.3]{HS-STNP1990})}
With notation as in~\eqref{eqn:definitions},
let \text{$v\in{M_\FF}$}, and let~\text{$w\in{M_\KK}$}
with~\text{$w\mid{v}$}.  Let~\text{$A\ge1$} be an integer, and
let
\[
Q_0,\ldots,Q_{6AN}\in{E(\KK_w)}
\]
be distinct points. Then there is a subset
\[
\{P_0,\ldots,P_N\} \subseteq \{Q_0,\ldots,Q_{6AN}\} 
\]
such that
\begin{equation}
  \label{eqn:lwPjPggeddd}
  \lhat_w(P_j-P_k) \ge \frac{1-A^{-1}}{12}\cdot \max\bigl\{v(j_E^{-1}),0\bigr\}
\quad\text{for all $j\ne k$.}
\end{equation}
\end{lemma}
\begin{proof}
If~$E$ has multiplicative reduction at~$w$, then
Lemma~\ref{lemma:stnp-prop1.3} is proven
in~\cite[Proposition~1.3]{HS-STNP1990}.
So we may suppose that~$E$ has good or additive reduction at~$w$.
\par
If~$E$ has potential good reduction, i.e., if~$v(j_E)\ge0$,
then~\eqref{eqn:lwPjPggeddd} is the assertion
that~$\lhat_w(P_j-P_k)\ge0$. This is an immediate consequence of
Lemma~\ref{lemma:potentialgoodreduction}(b), and any subset
of~$\{Q_0,\ldots,Q_{6AN}\}$ will suffice.
\par
If~$E$ has additive, potential multiplicative, reduction, then
Lemma~\ref{lemma:potentialgoodreduction}(d) implies that there is a
subgroup~$\G\subset{E(\KK_w)}$ of index~$2$ with the property
that\footnote{We could alternatively use the fact that the group of
components on the N{\'e}ron model has order~$4$, but that would lose a
factor of~$2$.}
\[
\lhat_w(P)\ge\frac12\Bernoulli(0)v(j_E^{-1})
= \frac{1}{12}v(j_E^{-1})
\quad\text{for all $P\in\G$.}
\]
We map
\[
\{Q_0,\ldots,Q_{2N}\} \longrightarrow E(\KK_w)/\G
\]
and use the pigeon-hole principle to find a subset
\[
\{P_0,\ldots,P_N\} \subseteq \{Q_0,\ldots,Q_{2N}\}
\quad\text{such that}\quad
P_j-P_k\in \G \quad\text{for all $j\ne k$.}
\]
This yields
\[
\lhat_w(P_j-P_k) \ge \frac{1}{12}v(j_E^{-1})
\quad\text{for all $j\ne k$,}
\]
which is stronger than the stated result.
\end{proof}

\begin{lemma}
\label{lemma:LS-Lemma9.4}
Let~$\a,\b,e_0,e_1,\ldots,e_r$ be positive real numbers satisfying
\[
e_0 = \max\{e_0,\ldots,e_r\}.
\]
Then
\begin{equation}
  \label{eqn:ae0sumboverei274}
  \a e_0 + \b \sum_{i=1}^r \frac{1}{e_i}
  \ge
  \biggl( \frac{27}{4}\a^2\b\sum_{i=0}^r e_i \biggr)^{1/3}.
\end{equation}
\end{lemma}
\begin{proof}
This was proven in~\cite{HS-STNP1990} with unspecified constants and
in~\cite[Lemma~9.4]{LS-TransAMS2024} without the~$27/4$ factor. We use
a different approach in Appendix~\ref{section:inquality} to
prove~\eqref{eqn:ae0sumboverei274}, as well as to eliminate an
annoying extra assumption required in the proof
of~\cite[Lemma~9.4]{LS-TransAMS2024}.  We also show that any
lower bound of the form~$C(\a,\b)(\sum{e_i})^{1/3+\e}$ with a constant that
is independent of~$r$ and~$e_0,\ldots,e_r$ necessarily
has~$\e\le0$, and that for~$\e=0$, we necessarily
have~$C(\a,\b)\le{C}(\a^2\b)^{1/3}$ for some absolute constant~$C$.
\end{proof}

\section{Proofs of the Main Theorems}
\label{sec:proofsmainthm}

\begin{proof}[Proof of Theorem~$\ref{theorem-2}$]
We set the notation
\begin{equation}
  \label{eqn:SigmaB}
  \Sigma(B) = \Sigma(B,E/\FF,\KK) := \bigl\{ P\in E(\KK) : \hhat_{E/\FF}(P) \le B \bigr\}.
\end{equation}
We choose a place~$v'\in{M_\FF}$ satisfying
\[
v'(j_E^{-1}) = \max_{v\in M_\FF} v(j_E^{-1}) \ge 1,
\]
where the lower bound follows from the assumption that~$j_E\notin{k}$.
Let
\[
w'_0,w'_1,\ldots,w'_r\in{M_\KK}
\]
be the places of~$K$ that lie over~$v'$, so in particular
\begin{equation}
  \label{eqn:Jwprimeiew}
  J_{w'_i} = e(w'_i)\cdot v'(j_E^{-1}) > 0.
\end{equation}
For convenience, we relabel the~$w'_i$ so that
\[
J_{w'_0} = \max_{0\le i\le r} J_{w'_i}.
\]
\par
We let
\[
B \quad\text{and}\quad N \ge M
\]
be parameters that we will specify later.
Suppose that we have distinct points
\[
Q_0,Q_1,\ldots,Q_{12N} \subset \Sigma(B).
\]
We apply Lemma~\ref{lemma:stnp-prop1.3} with~$v'$ and~$w'_0$
and~\text{$A=2$} to find distinct points
\[
P_0,\ldots,P_N \subseteq \{Q_0,Q_1,\ldots,Q_{12N}\}
\]
satisfying
\begin{equation}
  \label{eqn:stnp-prop1.3}
  \lhat_{w_0'}(P_j-P_k) \ge \frac{1}{24}\cdot v'(j_E^{-1}) > 0
\quad\text{for all $j\ne k$.}
\end{equation}
\par
Recalling that~$D=[\KK:\FF]$, we compute
\begin{align}
\sum_{\substack{j,k=0\\j\ne k\\}}^N & \hhat_{E/\FF}(P_j-P_k)  \\
&= \sum_{\substack{j,k=0\\j\ne k\\}}^N \frac{1}{D} \cdot \sum_{w\in M_\KK} e(w) \lhat_w(P_j-P_k)
\quad\text{from \eqref{eqn:hhatPsumlhatwP},} \notag\\
&\ge \sum_{\substack{j,k=0\\j\ne k\\}}^N \frac{1}{D} \cdot
\sum_{\substack{w\in M_\KK\\ w(j_E)<0\\}} e(w) \lhat_w(P_j-P_k)
\quad\TABT{from Lemma~\ref{lemma:potentialgoodreduction}(b), which\\
says that $\lhat_w\ge0$ if $w(j_E)\ge0$,\\} \notag\\
&= \frac{1}{D}
\sum_{\substack{w\in M_\KK\\ w(j_E)<0\\}} e(w)
\biggl( \sum_{\substack{j,k=0\\j\ne k\\}}^N \lhat_w(P_j-P_k) \biggr) \notag\\
&= \frac{1}{D}
\sum_{\substack{v\in M_\FF\\ v(j_E)<0\\}} 
\sum_{\substack{w\in M_\KK\\ w\mid v\\}} e(w)
\biggl( \sum_{\substack{j,k=0\\j\ne k\\}}^N \lhat_w(P_j-P_k) \biggr) \notag\\
&
\label{eqn:wsum1}
=\frac{1}{D} e(w'_0) \biggl( \sum_{\substack{j,k=0\\j\ne k\\}}^N \lhat_{w'_0}(P_j-P_k) \biggr) \\
&\label{eqn:wsum2}
\qquad{}+\frac{1}{D}\sum_{i=1}^r e(w'_i)
\biggl( \sum_{\substack{j,k=0\\j\ne k\\}}^N \lhat_{w'_i}(P_j-P_k) \biggr) \\
&\label{eqn:wsum3}
\qquad{}+\frac{1}{D}
\sum_{\substack{v\in M_\FF\\ v(j_E)<0\\ v\ne v'\\}} \sum_{\substack{w\in M_\KK\\ w\mid v\\}} e(w)
\biggl( \sum_{\substack{j,k=0\\j\ne k\\}}^N \lhat_w(P_j-P_k) \biggr).
\end{align}
\par
We use~\eqref{eqn:stnp-prop1.3} to estimate~\eqref{eqn:wsum1}, we use
Proposition~\ref{lemma:stnp-prop1.2} to estimate~\eqref{eqn:wsum2}, and we use
Proposition~\ref{lemma:stnp-prop1.2} without the positive term on the
right-hand side to estimate~\eqref{eqn:wsum3}. This yields
\begin{align}
  \sum_{\substack{j,k=0\\j\ne k\\}}^N & \hhat_{E/\FF}(P_j-P_k) \notag\\
  &\ge
  \frac{1}{D} e(w'_0) \sum_{\substack{j,k=0\\j\ne k\\}}^N \frac{1}{24}\cdot v'(j_E^{-1}) \notag\\
  &\qquad{}+\frac{1}{D}\sum_{i=1}^r
  e(w'_i) \left(\frac{1}{12}\left(\frac{N+1}{J_{w'_i}}\right)^2 v'(j_E^{-1})
  - \frac{N+1}{12}v'(j_E^{-1})\right) \notag\\
  &\qquad{}+\frac{1}{D}
  \sum_{\substack{v\in M_\FF\\ v(j_E)<0\\ v\ne v'\\}} \sum_{\substack{w\in M_\KK\\ w\mid v\\}} e(w)
  \cdot\left(- \frac{N+1}{12}v(j_E^{-1})\right) \notag\\
  &
  \label{eqn:wsum4}
  \ge \left(\frac{N^2+N}{24D}\cdot v'(j_E^{-1})\right) \cdot e(w'_0)
  \\*
  &
  \label{eqn:wsum5}
  \qquad{} +\frac{(N+1)^2}{12Dv'(j_E^{-1})}\cdot \sum_{i=1}^r \frac{1}{e(w'_i)} \\*
  &  \label{eqn:wsum6}
  \qquad{} - \frac{N+1}{12D}
  \sum_{\substack{v\in M_\FF\\ v(j_E)<0\\}} \sum_{\substack{w\in M_\KK\\ w\mid v\\}} e(w) \cdot v(j_E^{-1}) \\*
  &\omit\hfill\text{using~$J_{w'_i}=e(w'_i)\cdot{v'(j_E^{-1})}$ from~\eqref{eqn:Jwprimeiew}.}  \notag
\end{align}
We evaluate~\eqref{eqn:wsum6} (with minus sign omitted) as
\begin{align*}
  \frac{N+1}{12D}
  & \sum_{\substack{v\in M_\FF\\ v(j_E)<0\\}} \sum_{\substack{w\in M_\KK\\ w\mid v\\}} e(w) \cdot v(j_E^{-1}) \\
  &= \frac{N+1}{12D} \sum_{\substack{v\in M_\FF\\ v(j_E)<0\\}} D\cdot v(j_E^{-1})
  \quad\text{from~\eqref{eqn:sumeifi},} \\
  &= \frac{N+1}{12} h_\FF(j_E)
  \quad\TABT{from the definition~\eqref{eqn:defweilht} of height,\\
    where we note that $j_E^{-1}\in{\FF}$\\ and that $h_\FF(j_E^{-1})=h_\FF(j_E)$.\\}
\end{align*}
To ease notation, we let
\begin{equation}
  \label{eqn:alphabeta}
  \a = \frac{N^2+N}{24D} \cdot v'(j_E^{-1})
  \quad\text{and}\quad
  \b = \frac{(N+1)^2}{12Dv'(j_E^{-1})}.
\end{equation}
Then our inquality becomes
\begin{equation}
\label{eqn:sumjkgestuff1}
\sum_{\substack{j,k=0\\j\ne k\\}}^N \hhat_{E/\FF}(P_j-P_k) 
\ge \a e(w_0') + \b \sum_{i=1}^r \frac{1}{e(w'_i)} - \frac{N+1}{12} h_\FF(j_E).
\end{equation}
We are going to use Lemma~\ref{lemma:LS-Lemma9.4} to estimate the
first two terms, where in the lemma we take
\[
e_i=e_{w'_i},
\quad\text{so we have}\quad
e_0+\cdots+e_r=D.
\]
This yields
\begin{align}
\label{eqn:hPjPkN2N12}
\sum_{\substack{j,k=0\\j\ne k\\}}^N & \hhat_{E/\FF}(P_j-P_k) \notag\\
& \ge \left(\frac{27}{4}\a^2\b D\right)^{1/3} - \frac{N+1}{12} h_\FF(j_E)
\quad\text{from Lemma~\ref{lemma:LS-Lemma9.4} and~\eqref{eqn:sumjkgestuff1},}
\notag\\
&= \left( \frac{27}{4} \frac{(N^2+N)^2}{(24D)^2} \cdot v'(j_E^{-1})^2
\cdot \frac{(N+1)^2}{12Dv'(j_E^{-1})} \cdot D \right)^{1/3} - \frac{N+1}{12} h_\FF(j_E) \notag\\
&\omit\hfill\quad\text{using the values~\eqref{eqn:alphabeta} of~$\a$ and~$\b$,}\notag\\
&\ge (N+1)^2 \left( \left( \frac{N^2}{2^{10}\cdot(N+1)^2\cdot D^2} \right)^{1/3}  
- \frac{h_\FF(j_E)}{12(N+1)}  \right) \notag\\
&\omit\hfill\text{simplifying and using $v'(j_E^{-1})\ge1$,} \notag\\
&\ge (N+1)^2 \left( \left( \frac{1}{2^{5}\cdot(1+M^{-1})\cdot D} \right)^{2/3}  
- \frac{h_\FF(j_E)}{12(N+1)}  \right) \notag\\*
&\omit\hfill\text{simplifying and using  $N\ge{M}$.}
\phantom{\ref{eqn:hPjPkN2N12}\hspace*{2em}}
\end{align}
Hence 
\begin{align}
  \label{eqn:Bgt148hj}
  B
  &\ge\max_{0\le i\le 12N} \hhat_{E/\FF}(Q_i) \quad\text{since $Q_i\in\Sigma(B)$,} \notag\\
  &\ge \max_{0\le i\le N} \hhat_{E/\FF}(P_i)
  \quad\text{since $\{P_i\}_{0\le i\le N}\subset\{Q_i\}_{0\le i\le 12N}$} \notag\\
  &\ge \frac{1}{2(N+1)^2} \sum_{\substack{j,k=0\\j\ne k\\}}^N \hhat_{E/\FF}(P_j-P_k)
  \quad\text{from Lemma~\ref{lemma:hhatineq},} \notag\\
  &\ge \frac{1}{2} \left( \left( \frac{1}{2^{5}\cdot(1+M^{-1})\cdot D} \right)^{2/3}
  - \frac{h_\FF(j_E)}{12(N+1)}  \right)
  \quad\text{from \eqref{eqn:hPjPkN2N12}.}
\end{align}
\par
We let
\begin{equation}
  \label{eqn:BB148}
  B = B_{\d,M}
  := \frac{1-\d}{2} \left( \frac{1}{2^{5}\cdot(1+M^{-1})\cdot D} \right)^{2/3}.
\end{equation}
Then a bit of algebra with~\eqref{eqn:Bgt148hj} yields 
\begin{equation}
  \label{eqn:Nlt38hjd23}
  N + 1
  \le \frac{h_\FF(j_E)}{12\d} \left(\frac{2^5\cdot(1+M^{-1})\cdot D}{1}\right)^{2/3}.
\end{equation}
\par
To recapitulate, we have shown that for~$B_{\d,M}$ as given in~\eqref{eqn:BB148},
if
\[
\#\Sigma(B_{\d,M})\ge12N,
\]
for some~$N\ge{M}$, then~$N$ is bounded above
by~\eqref{eqn:Nlt38hjd23}. Using the specific values
in~\eqref{eqn:BB148} and~\eqref{eqn:Nlt38hjd23} yields
\begin{multline*}
  \#\Sigma\left(
  \frac{1-\d}{2} \left( \frac{1}{2^{5}\cdot(1+M^{-1})\cdot D} \right)^{2/3}
  \right) \\
  \le
  \max\left\{12M,\,
  \frac{h_\FF(j_E)}{\d} \left(2^5\cdot(1+M^{-1})\cdot D\right)^{2/3}\right\}.
\end{multline*}
Setting~$\e=M^{-1}$ completes the proof of Theorem~\ref{theorem-2}.
\end{proof}

\begin{proof}[Proof of Theorem~$\ref{theorem-1}$]
Continuing with the notation~\eqref{eqn:SigmaB} for~$\Sigma(B)$ as in
the proof of Theorem~\ref{theorem-2},
we write
\begin{align}
  \label{eqn:Cthm2122a}
  \Cl{thm2-1} &=\Cr{thm2-1}(\d,\e) :=
  \frac{1-\d}{2} \left( \frac{1}{2^{5}\cdot(1+\e)} \right)^{2/3}, \\
  \label{eqn:Cthm2122b}
  \Cl{thm2-2} &= \Cr{thm2-2}(\d,\e) :=
  \frac{1}{\d} \left(2^5\cdot(1+\e)\right)^{2/3},
\end{align} 
for the constants in Theorem~\ref{theorem-2} such that
\begin{equation}
  \label{eqn:thm2C1C2}
  \#\Sigma\left( \Cr{thm2-1} \cdot D^{-2/3} \right)
  \le \max\Bigl\{ 12\e^{-1},\,\Cr{thm2-2} \cdot h_\FF(j_E) \cdot D^{2/3} \Bigr\}.
\end{equation}
(We will choose~$\d$ and~$\e$ later.)
\par
\par
Let~$P\in{E(\KK)}$ be a non-torsion point,
and let~$\Cl{thm1-1}=\Cr{thm1-1}(\d,\e)$ denote a constant
whose value we will specify later. We suppose that
\begin{equation}
  \label{eqn:hEPleC1hf2D2}
  \hhat_{E/\FF}(P) \le \frac{\Cr{thm1-1}}{h_\FF(j_E)^2\cdot D^2}.
\end{equation}
Then for all~$m$ satisfying
\[
|m| \le (\Cr{thm2-1}/\Cr{thm1-1})^{1/2}\cdot h_\FF(j_E)\cdot D^{2/3},
\]
we have
\[
\hhat_{E/\FF}(mP) = m^2\hhat_{E/\FF}(P) \le \Cr{thm2-1} \cdot D^{-2/3}.
\]
Thus~\eqref{eqn:hEPleC1hf2D2} implies that
\begin{equation}
\label{eqn:setmPinSigma}
\bigl\{ mP : |m|\le  (\Cr{thm2-1}/\Cr{thm1-1})^{1/2} \cdot h_\FF(j_E)\cdot D^{2/3} \bigr\}
\subseteq \Sigma(\Cr{thm2-1}\cdot D^{-2/3}).
\end{equation}
\par
We compute
\begin{align*}
  2  (\Cr{thm2-1} /\Cr{thm1-1})^{1/2} &\cdot h_\FF(j_E)\cdot D^{2/3} - 1 \\
  &\le
  \#\Bigl\{ m\in\ZZ : |m|\le  (\Cr{thm2-1}/\Cr{thm1-1})^{1/2}
  \cdot h_\FF(j_E)\cdot D^{2/3} \Bigr\} \\*
  &\omit\hfill\text{using $2t-1\le2\lfloor{t}\rfloor+1$,} \\
  &\le
  \#\Bigl\{ mP : |m|\le  (\Cr{thm2-1}/\Cr{thm1-1})^{1/2}
  \cdot h_\FF(j_E)\cdot D^{2/3} \Bigr\} \\*
  &\omit\hfill\text{since $P$ is non-torsion, so the $mP$ are distinct,} \\
  &\le \#\Sigma(\Cr{thm2-1}\cdot D^{-2/3})
  \quad\text{from \eqref{eqn:setmPinSigma},}\\*
  &\le \max\Bigl\{ 12\e^{-1},\,\Cr{thm2-2} \cdot h_\FF(j_E)\cdot D^{2/3} \Bigr\}  
  \quad\text{from~\eqref{eqn:thm2C1C2}.}
\end{align*}
We thus obtain a contradiction to~\eqref{eqn:hEPleC1hf2D2}
if we choose~$\Cr{thm1-1}$ to satisfy
\[
\Cr{thm1-1} < \max\left\{
\frac{4\Cr{thm2-1}h_\FF(j_E)^2D^{4/3}}{(12\e^{-1}+1)^2},\;
\frac{4\Cr{thm2-1}}{\bigl(\Cr{thm2-2}+h_\FF(j_E)^{-1}\cdot D^{-2/3}\bigr)^2}
\right\}.
\]
Using the formulas~\eqref{eqn:Cthm2122a} and~\eqref{eqn:Cthm2122b}
for~$\Cr{thm2-1}$ and~$\Cr{thm2-2}$ and the fact that we may choose
any~$\d$ and~$\e$ between~$0$ and~$1$, we deduce (after some algebra)
that if we define
\begin{multline}
  \label{eqn:hEFPgemaxmin1}
  \Cl{maxmin}(\d,\e,J,D) =
  \left( \frac{1-\d}{2^{7/3}(1+\e)^{2/3}} \right) \\
  \cdot
  \min\left\{ \frac{ \e^2 J^2 D^{4/3}}{ (12+\e)^2 },\;
  \frac{\d^2}{ \bigl(2^{10/3} (1+\e)^{2/3}  + \d J^{-1} D^{-2/3}\bigr)^2 }
  \right\},
\end{multline}
then
\begin{equation}
  \label{eqn:hEFPgemaxmin2}
  \hhat_{E/\FF}(P) 
  \ge \max_{\substack{0<\d<1\\0<\e<1\\}} 
  \Cr{maxmin}\bigl(\d,\e,h_\FF(j_E),D\bigr)) \cdot\frac{1}{h_\FF(j_E)^2 D^2}.
\end{equation}
\par
We note that
\[
\Cr{maxmin}(\d,\e,J,D) \ge \Cr{maxmin}(\d,\e,1,1),
\]
so if we want a value of the constant in~\eqref{eqn:hEFPgemaxmin2}
that is valid for all~$E/\FF$ and all~$\KK/\FF$, we should
take~$(\d,\e)$ to maximize~$\Cr{maxmin}(\d,\e,1,1)$. There is
undoubtedly a clever way to maximize the
quantity~$\Cr{maxmin}(\d,\e,1,1)$ on the square~$[0,1]^2$,
but we simply did a brute-force search of a~$1000$-by-$1000$ point
grid in the square~$[0,1]^2$ to find the point~$(\d,\e)=(0.561,0.508)$
that yields the value
\[
\Cr{maxmin}(0.561,0.508,1,1) = 9.52455\cdot10^{-5} = \frac{1}{10499.2}.
\]
This gives a slightly better constant than quoted in the statement
of Theorem~\ref{theorem-1}.
\par
When~$h_\FF(j_E)$ or~$D$ is large, we can find an improved constant
by taking~$\e$ small and~$\d$ close to the value that maximizes~$(1-\d)\d^2$, i.e.,
taking~$\d$ close to~$\frac23$. For example, taking~$\e=\frac{1}{50}$ and~$\d=\frac23$
gives
\[
\Cr{maxmin}\left(\frac23,\frac{1}{50},J,D\right)
= \min\left\{ \frac{J^2D^{4/3}}{5.57\cdot10^6},\, \frac{1}{4107.37} \right\}.
\]
Hence
\[
J^2D^{4/3} \ge 1356.16
\quad\Longrightarrow\quad
\Cr{maxmin}\left(\frac23,\frac{1}{50},J,D\right)
=\frac{1}{4107.37}.
\]
This gives the second part of Theorem~\ref{theorem-1}.
\end{proof}

\section{Lehmer's Conjecture for Isotrivial Elliptic Curves}
\label{section:finalremarks0}
Conjecture~\ref{conjecture:lehmerecoverff} and
Theorems~\ref{theorem-1} and~\ref{theorem-2} all include the
assumption that~$E/\FF$ is not isotrivial, i.e., they assume
that~$j_E\notin{k}$. This is necessary, since the conjecture is false
in general if~$j_E\in{k}$.  To see why, suppose instead
that~$j_E\in{k}$. Then there is a finite extension~$\FF'/\FF$ and an
elliptic curve~$E'/k$ defined over the base field~$k$ such
that~$E/\FF'$ splits as a product
\[
E \times_{\Spec(\FF)} \Spec(\FF')\cong E' \times_{\Spec(k)} \Spec(\FF').
\]
Then for any extension~$\KK/\FF'$, the group~$E(\KK)$ contains a subgroup that
is isomorphic to~$E'(k)$, and all of the points in this subgroup have
canonical height~$0$. Thus
Conjecture~\ref{conjecture:lehmerecoverff} is not true
if~$j_E\in{k}$.
\par
We can rescue Conjecture~\ref{conjecture:lehmerecoverff} by
restricting to extensions~$\KK/\FF$ over which~$E/\KK$ does not split,
or in fancier language, extensions such that the~$\KK/k$-trace of~$E$
is trivial.

\begin{proposition}
\label{proposition:isotrivialhtgeD}
Suppose that~$E/\FF$ is isotrivial, and for~$\KK/\FF$,
let~$\Dcal_{E/\KK}$ be the minimal discriminant of~$E/\KK$.
\begin{parts}
\Part{(a)}
Every~$P\in{E(\KK)}$ satisfies either
\[
12P=0
\quad\text{or}\quad
\hhat_{E/\FF}(P) \ge \frac{\deg(\Dcal_{E/\KK})}{12^3[\KK:\FF]}.
\]
\Part{(b)}
If~$E$ does not split over~$\KK$,
i.e., if~$j_E\in{k}$ and~$\Trace_{\KK/k}(E)=0$,
then~$\deg(\Dcal_{E/\KK})\ge1$, so the lower bound in
Conjecture~\ref{conjecture:lehmerecoverff} is valid
in such extensions~$\KK/\FF$.
\end{parts}
\end{proposition}
\begin{proof}
(a)\enspace
The assumption that~$E$ is isotrivial implies
that~$E/\KK$ has only additive reduction.
Hence for~$P\in{E(\KK)}$ satisfying~$12P\ne0$, we have
\begin{align*}
  \hhat_{E/\FF}(P) &= \frac{1}{12^2} \hhat_{E/\FF}(12P) \\
  &= \frac{1}{12^2D} \sum_{w\in M_\KK} e(w)\lhat_w(12P)
  \quad\text{from Theorem~\ref{theorem:localhtexist}(b),}\\
  &\ge \frac{1}{12^2D} \sum_{w\in M_\KK} \frac{e(w)}{12}w(\Dcal_{E/\KK})
  \quad\text{from Lemma~\ref{lemma:potentialgoodreduction}(a),} \\
  &= \frac{\deg(\Dcal_{E/\KK})}{12^3D}.
\end{align*}
This proves the first part of Proposition~\ref{proposition:isotrivialhtgeD}.
For the second part we note that
\[
\Trace_{\KK/k}(E)=0
\quad\Longleftrightarrow\quad
\Dcal_{E/\KK}\ne0
\quad\Longleftrightarrow\quad
\deg(\Dcal_{E/\KK})\ge 1.
\qedhere
\]
\end{proof}

\appendix

\section{Proof of an Inequality}
\label{section:inquality}

In this section we prove Lemma~\ref{lemma:LS-Lemma9.4}, which we
recall says that for all positive real
numbers~$\a,\b,e_0,e_1,\ldots,e_r$, we have
\[
  e_0 = \max\{e_0,\ldots,e_r\}
  \quad\Longrightarrow\quad
  \a e_0 + \b \sum_{i=1}^r \frac{1}{e_i}
  \ge
  \biggl( \frac{27}{4}\a^2\b\sum_{i=0}^r e_i \biggr)^{1/3}.
\]

\begin{proof}[Proof of Lemma~$\ref{lemma:LS-Lemma9.4}$]
We denote the all-ones vector by
\[
\ONE_n = (1,1,\ldots,1) \in \ZZ^n.
\]
For~$\bfx=(x_0,\ldots,x_r)\in\RR_{>0}^{r+1}$ and~$\a,\b>0$ and~$\e\ge0$, we define
\[
f_\e(\a,\b,\bfx) = \frac{ \a x_0 + \b x_1^{-1} + \cdots + \b x_r^{-1} }{ (x_0+x_1+\cdots+x_r)^{1/3+\e} }.
\]
The goal is to find a lower bound for~$f_\e(\a,\b,\bfx)$ that depends only on~$\a$,~$\b$, and~$\e$, but
is independent of both~$r$ and~$\bfx$.
\par
We note that
\begin{align*}
f_\e(\a,\b,\sqrt{2\b r/\a} \cdot \ONE_{r+1})
&= \frac{ \sqrt{2\a\b r} + r\sqrt{\a\b /2 r} }{ \bigl((r+1)\sqrt{2\b r/\a} \bigr)^{1/3+\e} } \\
&= \frac{ (2^{1/2}+2^{-1/2})\sqrt{\a\b r} }
{ r^{1/2+3\e/2} \cdot (2\b/\a)^{1/6+\e/2} \cdot (1+r^{-1})^{1/3+\e}} \\
&\xrightarrow{\;\;r\to\infty\;\;}
\begin{cases}
  3(\a^2\b/4)^{1/3} &\text{if $\e=0$.} \\
  0 &\text{if $\e>0$.} \\
\end{cases}
\end{align*}
This shows that we cannot hope for a lower bound
in~\eqref{eqn:ae0sumboverei274} that is independent of~$r$ using an
exponent strictly larger than~$1/3$, and that for exponent~$1/3$, the
best possible lower bound is~$3(\a^2\b/4)^{1/3}$, where we have used the
simplification \text{$(2^{1/2}+2^{-1/2})\cdot2^{-1/6}=3\cdot2^{-2/3}$}.
\par
We henceforth let
\[
f(\bfx) := f_0(\a,\b,\bfx) := 
\frac{ \a x_0 + \b x_1^{-1} + \cdots + \b x_r^{-1} }{ (x_0+x_1+\cdots+x_r)^{1/3} }.
\]
To ease notation, we let
\[
A = A(\bfx) := \a x_0 + \b\sum_{i=1}^r\frac{1}{x_i}
\quad\text{and}\quad
D = D(\bfx) := \sum_{i=0}^r x_i,
\]
so
\[
f(\bfx) = \frac{A(\bfx)}{D(\bfx)^{1/3}}
\]
We want to minimize~$f(\bfx)$ on the region
\[
\Rcal := \Bigl\{ \bfx : \text{$0\le x_i \le x_0$ for all $0\le i\le r$}\Bigr\}.
\]
We note that as any coordinate~$x_j\to0$, we have~$f(\bfx)\to\infty$, so the minimum does
not occur along any of the coordinate hyperplanes. We first look for
critical points in the interior of the region~$\Rcal$. For any~$1\le{j}\le{r}$ we
have
\begin{equation}
\label{eqn:partialfxxj}
\frac{\partial f(\bfx)}{\partial x_j}
= \frac{D^{1/3} \b (-x_j^{-2}) - A \cdot\frac13\cdot D^{-2/3}}{D^{2/3}}
= -\frac{3 D \b x_j^{-2} + A}{3 D^{4/3}}.
\end{equation}
All of the quantities~$A$,~$D$,~$\b$, and~$x_j^2$ are strictly positive, so
$\partial f(\bfx)/\partial x_j$ doesn't vanish in the interior of~$\Rcal$.
Hence the minimum is attained on the boundary of~$\Rcal$.
\par
The boundary of~$\Rcal$ is the union of~$r$ regions,
one region for each \text{$1\le k\le r$}, having the form
\[
\Bigl\{ \bfx : \text{$0\le x_i \le x_0 = x_k$ for all $0\le i\le r$} \Bigr\}.
\]
By symmetry, it suffices to take~$k=1$. More generally, again by symmetry, it suffices to
look for the minimum of~$f(\bfx)$ in the interior of each of the regions
\[
\Rcal_s := \Bigl\{ \bfx :
\text{$0\le x_i \le x_0 = x_1 = \cdots = x_s$ for all $0\le i\le r$} \Bigr\}.
\]
To ease notation, we write points in~$\Rcal_s$ as
\[
(\bfx,\bfy) :=
(\underbrace{x,x,\ldots,x}_{\hidewidth\text{$s+1$ copies of $x$}\hidewidth},y_1,\ldots,y_t)
\quad\text{with $t=r+1-s$.}
\]
For~$t\ge1$, if we look for critical points in the interior
of~$\Rcal_s$, we find that\footnote{We note that if~$t=r$, then this
is the same calculation~\eqref{eqn:partialfxxj} that we did earlier.}
\[
\frac{\partial f(\bfx)}{\partial y_1}
= \frac{D^{1/3} \b (-y_1^{-2}) - A \cdot\frac13\cdot D^{-2/3}}{D^{2/3}}
= -\frac{3 D \b y_1^{-2} + A}{3 D^{4/3}}.
\]
This shows there are no critical points in the interior of~$\Rcal_s$, which proves
that the minimum value of~$f(\bfx)$ is attained in the region~$\Rcal_{r+1}$ with~$t=0$.
In other words, the minimum is attained at a point of the form
\[
x\ONE_{r+1} = (x,x,x,\ldots,x).
\]
\par
We are thus reduced to studying
\[
f(x\ONE_{r+1}) = \frac{\a x + \b r/x}{  (r+1)^{1/3}x^{1/3} }.
\]
A short calculation yields
\begin{align*}
\frac{d f(x\ONE_{r+1})}{dx}
&= \frac{2(\a x^2 - 2\b r)}{3(r+1)^{1/3}x^{7/3}} .
\end{align*}
There is thus a unique critical point at~$x=\sqrt{2\b{r}/\a}$, which is necessarily
a minimum since~$f(x\ONE_{r+1})\to\infty$ as~$x\to0^+$ and as~$x\to\infty$.
Hence
\begin{align*}
\inf_{\bfx\in\Rcal} f(\bfx)
&= \inf_{x\in\Rcal_{r+1}} f(x\ONE_{r+1})\\
&= f\bigl(\sqrt{2\b r/\a}\,\ONE_{r+1}\bigr)\\
&= \frac{\a (2\b r/\a)^{1/2} + \b r  (2\b r/\a)^{-1/2}}{(r+1)^{1/3}\cdot  (2\b r/\a)^{1/6}}\\
&=(2^{1/2}+2^{-1/2})\cdot\frac{1}{2^{1/6}}\cdot\left(\frac{\a^2\b}{1+r^{-1}}\right)^{1/3}.
\end{align*}
We now let~$r\to\infty$ to obtain a lower bound that holds for
all~$r$, and we rewrite the constant
using~\text{$(2^{1/2}+2^{-1/2})\cdot2^{-1/6}=(27/4)^{1/3}$} to obtain the desired
result.
\end{proof}

\section{Local Height Formulas via Weierstrass Equations}
\label{section:localhtImstar}

In this section we derive formulas for the local height in the case
of~$I_M^*$ reduction. We remark that these formulas appear already
in~\cite[Table~1]{JHSThesis}. They are also implicit
in~\cite[Table~1.19]{CoxZucker} in the case that~$\FF$ is the function
field of a curve over~$\CC$. But for completeness, we give here a
proof using explicit Weierstrass equations that is valid in all
residue characteristics other than~$2$ and~$3$, and in
Section~\ref{section:localhtImstarintersection} we give an alternative
proof using intersection theory on the associated elliptic surface.
\par
Let~$(\FF_v,v)$ be a complete local field with  valuation normalized so that
\text{$v:\FF_v^*\onto\ZZ$}, and let~$E/\FF_v$ be an elliptic curve with
additive, potential multiplicative, reduction. Thus in the
Kodaira--N{\'e}ron
classification~(\cite{Antwerp4Tate},~\cite[Table~15.1]{AEC}), the
elliptic curve~$E$ has reduction type~$I_M^*$ for some~$M\ge1$.  The
special fiber of a complete minimal model of~$E$ at~$v$ is illustrated
in Figure~\ref{figure:IMstar}.  The fiber consists of four components
of multiplicity~$1$ that we have
labeled~$\boldsymbol0$,~$\boldsymbol\a$,~$\boldsymbol\b$,
and~$\boldsymbol{\b'}$, together with~\text{$m+1$} components of
multiplicity~$2$. Points in~$E_0(\FF_v)$ are the points that intersect
the~$\boldsymbol0$-component. More generally, we let
\[
E_\nu(\FF_v) := \bigl\{P\in{E}(\FF_v):v\bigl(x(P)^{-1}\bigr)\ge2\nu\bigr\}
\]
denote the points in the $\nu$th level of the formal group filtration,
where~$x$ is a coordinate on a minimal Weierstrass equation for~$E$
at~$v$.  The group of components~$E(\FF_v)/E_0(\FF_v)$ is
either~$(\ZZ/2\ZZ)^2$ of~$(\ZZ/4\ZZ)$ depending on whether~$M$ is even
or odd, with the $\boldsymbol\b$ and $\boldsymbol{\b'}$-components
being the elements of order~$4$ in the latter case.  We also note that
the minimal discriminant of~$E$ at~$v$ satisfies
\begin{equation}
  \label{eqn:IMstarstuff2}
  M = v(\Dcal_{E/KK}) - 6 = v(j_E^{-1}).
\end{equation}

\begin{proposition}
\label{proposition:localhtImstar}
With notation as above for an elliptic curve with
Type~$I_M^*$ reduction, the canonical local height
\[
\lhat_v : E(\FF_v)\setminus0 \longrightarrow [0,\infty)
\]
is given by  
\[
\lhat_v(P) = \begin{cases}
  \phantom{-}\dfrac{v(j_E^{-1})}{12}+\nu
  &\text{if $P\in E_\nu(\FF_v)\setminus E_{\nu+1}(\FF_v)$,} \\[2\jot]
  \phantom{-}\dfrac{v(j_E^{-1})}{12}= \frac12\Bernoulli_2(0)v(j_E^{-1})
  &\text{if $P$ hits the $\boldsymbol\a$-component,} \\[2\jot]
  -\dfrac{v(j_E^{-1})}{24}= \frac12\Bernoulli_2(\frac12)v(j_E^{-1})
  &\text{if $P$ hits the $\boldsymbol\b$ or $\boldsymbol{\b'}$-component.} \\
\end{cases}
\]
\end{proposition}

\begin{figure}[t]
  \begin{picture}(300,100)
    \setlength{\unitlength}{0.5pt}
    \thicklines
    \put(20,30){\line(0,1){100}}  \put(14,135){\makebox(0,0)[b]{$\boldsymbol0$}}
    \put(50,30){\line(0,1){100}}  \put(56,135){\makebox(0,0)[b]{$\boldsymbol\a$}}
    \put(0,100){\line(1,-1){100}}
    \put(75,0){\line(1,1){100}}
    \put(125,100){\line(1,-1){100}}
    \multiput(225,20)(20,0){5}{\circle*{3}}
    \put(500,30){\line(0,1){100}}  \put(494,135){\makebox(0,0)[b]{$\boldsymbol\b$}}
    \put(530,30){\line(0,1){100}}  \put(536,135){\makebox(0,0)[b]{$\boldsymbol{\b'}$}}
    \put(550,100){\line(-1,-1){100}}
    \put(475,0){\line(-1,1){100}}
    \put(425,100){\line(-1,-1){100}}
    \put(70,38){\makebox(0,0)[b]{$\bfgamma_0$}}
    \put(133,40){\makebox(0,0)[t]{$\bfgamma_1$}}
    \put(203,50){\makebox(0,0)[t]{$\bfgamma_2$}}
    \put(350,50){\makebox(0,0)[tr]{$\bfgamma_{M-2}$}}
    \put(445,35){\makebox(0,0)[tr]{$\bfgamma_{M-1}$}}
    \put(485,40){\makebox(0,0)[br]{$\bfgamma_{M}$}}
  \end{picture}
  \caption{The special fiber for $I_M^*$ reduction. It consists of~$4$
    fibers of multiplicity~$1$,
    labeled~$\boldsymbol0$,~$\boldsymbol\a$,~$\boldsymbol\b$
    and~$\boldsymbol{\b'}$, and~$M+1$ components of multiplicity~$2$
    labeled $\bfgamma_0,\ldots,\bfgamma_M$.}
  \label{figure:IMstar}
\end{figure}

\begin{proof}
By a minor abuse of notation, we let
\begin{align*}
  E^{(\bfzero)}(\FF_v)
  &:= \bigl\{ P\in E(\FF_v) : \text{$P$ hits the $\bfzero$-component} \bigr\}, \\
  E^{(\bfalpha)}(\FF_v)
  &:= \bigl\{ P\in E(\FF_v) : \text{$P$ hits the $\bfalpha$-component} \bigr\}, \\
  E^{(\bfbeta)}(\FF_v)
  &:= \bigl\{ P\in E(\FF_v) : \text{$P$ hits the $\bfbeta$ or $\bfbeta'$-component} \bigr\}.
\end{align*}
\par
Let~$\pi$ be a uniformizer for~$v$.
The assumption that~$\FF_v$ does not have characteristic~$2$ or~$3$
means that we can start with a Weierstrass
equation in short normal form. Applying Tate's algorithm
(\cite{Antwerp4Tate},~\cite[IV~\S9]{ATAEC}) along the branch that leads
to Type~$I_M^*$ shows that
\[
E:y^2=x^3+a_4 x+a_6
\]
has Type~$I_M^*$ reduction if and only if
\[
v(a_4)=2,\quad v(a_6)=3,\quad
v(\D) = v\bigl(-2^4\cdot(4a_4^3+27a_6^2)\bigr) = 6+M.
\]
We let~$a_4=-3\pi^2a$ and~$a_6=2\pi^3b$, giving the equation
\[
E:y^2=x^3-3\pi^2ax+2\pi^3b,
\]
with
\[
v(a)=v(b)=0,
\quad
v(\D)=v\bigl(12^3\cdot\pi^6\cdot(a^3-b^2)\bigr)=6+M.
\]
Since~$\FF_v$ is complete with algebraically closed residue field,
Hensel's lemma tells us that~$\a\in\FF_v^*$ is a square if and only
if~$v(\a)$ is even. In particular, we can find an~$A\in\FF_v^*$
satisfying\footnote{The equality $v(A^3-b)=M$ serves to fix the sign
of~$A$, since we know that~$v(A^6-b^2)=M$, and only one of~$A^3\pm{b}$
has positive valuation.}
\[
A^2=a\quad\text{and}\quad v(A^3-b)=M.
\]
We change variables
\[
x\to{x+\pi{A}}
\quad\text{and for convenience we set}\quad
D=A^3-b.
\]
This yields the minimal Weierstrass equation
\[
E : y^2 = x^3 - 3\pi A x^2 - 2\pi^3 D
\]
with
\begin{align*}
v(A)&=0, \quad v(D)=M=v(j_E^{-1}),\\
\Dcal_E &= 12^3 \pi^6 D (2A^3-D), \\
v(\Dcal_E) &= 6+v(D) = 6+M.
\end{align*}
For a Weierstrass equation of this form, the duplication formula~\cite[III.2.3(d)]{AEC}
takes the form
\[
x(2P) = \frac{x^4 + 16 \pi^3 D x + 24 \pi^4 D A}{4y^2}.
\]
\par
Let~$P_0=(x_0,y_0)\in{E(\FF_v)}$. Tracing the points through Tate's
algorithm, we find that:
\begin{align*}
  P_0 \in E^{(\bfzero)}(\FF_v)
  &\;\Longleftrightarrow\;
  v\bigl(x(P_0)\bigr) \le 0, \\
  P_0 \in E^{(\bfalpha)}(\FF_v)
  &\;\Longleftrightarrow\;
  v\bigl(x(P_0)\bigr) = 1, \\
  \text{$P_0 \in E^{(\bfbeta)}(\FF_v)$ and $M$ odd}
  &\;\Longleftrightarrow\;
  \text{$v\bigl(x(P_0)\bigr) \ge \frac{M+3}{2}$ and $v\bigl(y(P_0)\bigr) = \frac{M+3}{2}$,} \\
  \text{$P_0 \in E^{(\bfbeta)}(\FF_v)$ and $M$ even}
  &\;\Longleftrightarrow\;
  \text{$v\bigl(x(P_0)\bigr) = \frac{M+2}{2}$ and $v\bigl(y(P_0)\bigr) \ge \frac{M+4}{2}$.} 
\end{align*}
\par
\Case{$\bfalpha$}{$P_0\in{E^{(\bfalpha)}(\FF_v)}$}
In this case we always have~$2P_0\in{E^{(\bfzero)}}$.
The three monomials in the
numerator of~$x(2P_0)$ have valuations
\[
v(x_0^4)=4,\quad
v(16\pi^3 D x_0)=4+M,\quad
v(24\pi^4 D A) = 4+M,
\]
so the valuation of the numerator of~$x(2P_0)$ is~$4$. Since~$2P_0\in{E^{(\bfzero)}}(\FF_v)$,
its local height is given by
\begin{align}
  \label{eqn:lv2P0IMstar}
  \lhat_v(2P_0)
  &= \frac12 v\bigl(x(2P_0)^{-1}\bigr) + \frac{1}{12}v(\Dcal_E) \notag\\*
  &= \frac12v(y_0^2) - \frac12\cdot 4 + \frac{1}{12}v(\Dcal_E) \notag\\*
  &= v(y_0) - 2 + \frac{1}{12}v(\Dcal_E).
\end{align}
We use this to compute
\begin{align*}
\lhat_v(P_0)
&= \frac14\Bigl( \lhat_v(2P_0) - v(2y_0) + \frac14v(\Dcal_E) \Bigr)
\quad\text{duplication formula,} \\
&= \frac14 \left(
\left\{
v(y_0) - 2 + \frac{1}{12}v(\Dcal_E)
\right\} - v(2y_0) + \frac14v(\Dcal_E) \right)
\quad\text{from \eqref{eqn:lv2P0IMstar},} \\
&= \frac{1}{12}v(\Dcal_E) - \frac12 \quad\text{note that $v(2)=0$,}\\
&= \frac{v(j_E^{-1})}{12}\quad\text{since $v(\Dcal_E)=M+6=v(j_E^{-1})+6$.}
\end{align*}
\par
\Case{$\bfbeta$-even}{$P_0\in{E^{(\bfbeta)}(\FF_v)}$ and $M$ even}
This case is characterized by
\[
v(x_0)=\frac{M+2}{2}\quad\text{and}\quad v(y_0)\ge\frac{M+4}{2}.
\]
Hence the three monomials in the
numerator of~$x(2P_0)$ have valuations
\[
v(x_0^4)=2M+6,\quad
v(16\pi^3 D x_0)
=\frac32M+4,\quad
v(24\pi^4 D A) = 4+M.
\]
Since
\[
4+M < \frac32M+4 < 2M+6 \quad\text{for all $M\ge1$,}
\]
we find that
\[
v\bigl(x(2P_0)\bigr) = 4+M-2v(y_0)
= v(\Dcal_E)-2-2v(y_0). 
\]
Further, since~$E(\FF_v)/E_0(\FF_v)=(\ZZ/2\ZZ)^2$,
we  again have~$2P_0\in{E^{(\bfzero)}}(\FF_v)$,
so we can compute its local height as
\begin{align}
  \label{eqn:lv2P0IMstarbetaeven}
  \lhat_v(2P_0)
  &= \frac12 v\bigl(x(2P_0)^{-1}\bigr) + \frac{1}{12}v(\Dcal_E) \notag\\*
  &= \frac12\Bigl(-v(\Dcal_E)+2+2v(y_0)\Bigr) + \frac{1}{12}v(\Dcal_E) \notag\\*
  &= v(y_0) + 1 - \frac{5}{12}v(\Dcal_E).
\end{align}
We use this to compute
\begin{align*}
\lhat_v(P_0)
&= \frac14\Bigl( \lhat_v(2P_0) - v(2y_0) + \frac14v(\Dcal_E) \Bigr)
\quad\text{duplication formula,} \\
&= \frac14 \left(
\left\{
v(y_0) + 1 - \frac{5}{12}v(\Dcal_E)
\right\} - v(2y_0) + \frac14v(\Dcal_E) \right)
\quad\text{from \eqref{eqn:lv2P0IMstarbetaeven},} \\
&=  -\frac{1}{24}v(\Dcal_E) + \frac14 \quad\text{note that $v(2)=0$,}\\
&= -\frac{v(j_E^{-1})}{24}\quad\text{since $v(\Dcal_E)=M+6=v(j_E^{-1})+6$.}
\end{align*}
\par
\Case{$\bfbeta$-odd}{$P_0\in{E^{(\bfbeta)}(\FF_v)}$ and $M$ odd} 
This case is characterized by
\[
v(x_0)\ge\frac{M+3}{2}\quad\text{and}\quad v(y_0)=\frac{M+3}{2}.
\]
Further, we have~$E(\FF_v)/E_0(\FF_v)\cong\ZZ/4\ZZ$, with the~$\bfbeta$ components
being the elements of order~$4$, and thus~$2P_0\in{E^{(\bfalpha)}}$.
It follows from a case that we already did that
\[
\lhat(2P_0) = \frac{v(j_E^{-1})}{12}.
\]
Hence
\begin{align}
  \label{eqn:lv2P0IMstarbetaodd}
  \lhat_v(P_0)
  &= \frac14\Bigl( \lhat_v(2P_0) - v(2y_0) + \frac14v(\Dcal_E) \Bigr)
  \quad\text{duplication formula,} \notag\\
  &= \frac14\left( \frac{v(j_E^{-1})}{12} - \frac{M+3}{2} + \frac14v(\Dcal_E) \right) \notag\\
  &= \frac14\left( \frac{v(j_E^{-1})}{12} - \frac{v(j_E^{-1})+3}{2} + \frac14\Bigl(
  6+v(j_E^{-1})\Bigr)\right) \notag\\
  &= -\frac{1}{24}v(j_E^{-1}).
  \qedhere
\end{align}
\end{proof}

\section{Local Height Formulas via Intersection Theory}
\label{section:localhtImstarintersection}
An alternative approach to canonical heights in the function field
setting is via intersection theory, an idea appearing already
in~\cite{Manin}.  Writing a global intersection as a sum of local
intersections gives another way to derive local height formulas, as is
done in~\cite{CoxZucker} for elliptic surfaces over~$\CC$.  We 
sketch the proof and rederive the local height formula for~$I_M^*$ reduction,
omitting some details.
\par
Let~$\pi:\Ecal\to{C}$ be a smooth minimal elliptic fibration over a
smooth projective curve, everything defined over an algebraically
closed field~$k$. For this section we let~$\FF=k(C)$ be the global
function field, and for~$v\in{M_\FF}$, i.e., for points of~$C$, we
denote the completion of~$\FF$ at~$v$ by~$\FF_v$.
Let~$\Kcal_\Ecal\in\Pic(\Ecal)$ be the canonical bundle, and
let~$\Dcal_\Ecal\in\Div(C)$ be the minimal discriminant of~$\Ecal/C$.
Fixing a Weierstrass equation~$y^2=f(x)$ and using the
section~$(dx)^6/\D$ to the~$12$th-tensor-power
$\Omega_\Ecal^{\otimes12}$ of the sheaf of relative
differentials~$\Omega_\Ecal$ on~$\Ecal$, a local computation yields
\begin{equation}
  \label{eqn:OEpiDEKEC12}
  \Ocal_\Ecal(\pi^*\Dcal_\Ecal) \cong \Kcal_{\Ecal/C}^{\otimes12},
\end{equation}
where $\Kcal_{\Ecal/C}:=\Kcal_\Ecal\otimes\pi^*\Kcal_C^{-1}$
is the relative canonical bundle for the fibration~$\pi:\Ecal\to{C}$.
\par
Let~$O:C\to\Ecal$ be the zero-section, and let~$P:C\to\Ecal$ be an
arbitrary section. We identify~$P$ and~$O$ with their images
in~$\Ecal$, so in particular~$(P)$ and~$(O)$ are divisors on~$\Ecal$.
The Hodge index theorem, i.e., the fact that the intersection pairing
is negative definite except for one positive direction, allows one to
find a fibral divisor~$\F_P\in\Div(\Ecal)$ so that the
divisor\footnote{The divisor~$\F_P$ is uniquely determined up to
addition of a divisor of the form~$\pi^*\G$ for
some~$\G\in\Div(C)$. We can pin down a particular~$\F_P$ by requiring
that~$i_v\bigl(\F_P,(O)\bigr)=0$ for all~$v$, i.e., the support~$\F_P$
does not contain any fibral components that intersect the zero
section; but any choice of~$\F_P$ will yield the same formulas
for~$\hhat_E$ and~$\lhat_v$.}
\[
D_P:=(P) - (O) + \F_P
\]
satisfies
\[
D_P \cdot F = 0\quad\text{for all fibral divisors $F\in\Div(\Ecal)$.}
\]
Then one can show that
\begin{equation}
  \label{eqn:hecalP}
  \hhat_\Ecal(P) = -\frac12 D_P^2.
\end{equation}
\par
We observe that the translation-by-$P$ automorphism~$\t_P:\Ecal\to\Ecal$
satisfies~$\t_P^*(P)=(O)$, so~$(P)^2=(O)^2$. We compute
\begin{align}
  \label{eqn:Osqr2112}
  (O)^2
  &= -(O)\cdot\Kcal_\Ecal + 2g(C)-2 \quad\text{the adjunction formula,} \notag\\
  &= -(O)\cdot\Kcal_\Ecal + \deg\Kcal_C \notag\\
  &= -(O)\cdot\Kcal_\Ecal + (O)\cdot\pi^*\Kcal_C \notag\\
  &= -(O)\cdot \left( \Kcal_\Ecal\otimes \pi^*\Kcal_C^{-1} \right) \notag\\
  &= -(O)\cdot \Kcal_{\Ecal/C} \notag\\
  &= -\frac{1}{12} (O)\cdot\pi^*\Dcal_\Ecal \quad\text{from \eqref{eqn:OEpiDEKEC12},}\notag\\
  &= -\frac{1}{12}\deg\Dcal_\Ecal.  
\end{align}
Putting this together, we find that
\begin{align*}
\hhat_\Ecal(P)
&= -\frac12 D_P^2 \quad\text{from \eqref{eqn:hecalP},} \\
&= -\frac12 D_P \cdot \bigl( (P) - (O) + \F_P \bigr)
\quad\text{by definition of $D_P$,}\\
&= -\frac12 D_P \cdot \bigl( (P) - (O) \bigr)
\quad\TABT{since $D_P\cdot F=0$ for fibral $F$,\\ and $\F_P$ is fibral,\\} \\
&= -\frac12 \bigl( (P) - (O) + \F_P \bigr) \cdot \bigl( (P) - (O) \bigr)
\quad\text{by definition of $D_P$,}\\
&= (P)\cdot(O) -\frac12(P)^2 -\frac12(O)^2 - \frac12\F_P\cdot\bigl( (P) - (O) \bigr) \\
&= (P)\cdot(O) + \frac{1}{12}\deg\Dcal_\Ecal - \frac12 \F_P\cdot\bigl( (P) - (O) \bigr) 
\quad\text{from \eqref{eqn:Osqr2112}.} 
\end{align*}
For~$P\ne{O}$, this allows a local decomposition by setting
\begin{equation}
  \label{eqn:lhatvPivPOivFPPO}
  \lhat_v(P) := i_v\bigl( (P),(O) \bigr) + \frac{1}{12}v(\Dcal_\Ecal)
  - \frac12 i_v\bigl( \F_P, (P)-(O) \bigr).
\end{equation}
\par
We note that if~$P\in E_0(\FF_v)$, i.e., if~$P$ goes through
the identity component of the fiber over~$v$, then the fact that~$\F_P$ is
fibral implies that the last term vanishes. Further, working
on a minimal Weierstrass equation, the first term 
\[
i_v\bigl( (P),(O) \bigr)
= \frac12\max\bigl\{ v\bigl(x(P)^{-1}\bigr),\,0\bigr\}
\]
is the location of~$P$ in the formal group filtration. Hence
\[
P\in E_0(\FF_v)
\;\Longrightarrow\;
\lhat_v(P) = \nu  + \frac{1}{12}v(\Dcal_\Ecal)
\quad\text{for $P\in E_\nu(\FF_v)\setminus E_{\nu+1}(\FF_v)$.} 
\]
\par
Finally, for~$P\notin{E_0(\FF_v)}$, the local height~$\lhat_v(P)$
depends only on the fibral component hit by~$P$ and can be determined
by computing the divisor~$\F_P$.  To do this, we start by computing
the intersection matrix~$A$ for the components of the special
fiber. Figure~\ref{figure:IMstarintersectionmatrix} gives the matrix~$A(I_M^*)$
for the~$I_M^*$ configuration illustrated in
Figure~\ref{figure:IMstar}.

\begin{figure}[t]
\[
\bordermatrix{
  & 0 & \bfalpha & \bfgamma_0 & \bfgamma_1 & \cdots & \cdots
  & \bfgamma_{M-1} & \bfgamma_M & \bfbeta & \bfbeta' \cr
  0          & -2 & 0 & 1  \cr
  \bfalpha   & 0 & -2 & 1  \cr
  \bfgamma_0 & 1 & 1 & -2 & 1  \cr
  \bfgamma_1 & & & 1 & -2 & 1 \cr
  \vdots     & & & & \ddots & \ddots & \ddots\cr
  \vdots     & & & & & \ddots & \ddots & \ddots\cr
  \bfgamma_{M-1} & & & & & & 1 & -2 & 1 \cr
  \bfgamma_M & & & & & & & 1 & -2 & 1 & 1\cr
  \bfbeta    & & & & & & & & 1 & -2 & 1\cr
  \bfbeta'   & & & & & & & & 1 & 0 & -2\cr
  }
\]
\caption{The intersection matrix  $A(I_M^*)$ for the $I_M^*$ components
  illustrated in Figure~\ref{figure:IMstar}.}
\label{figure:IMstarintersectionmatrix}
\end{figure}

The null space of~$A(I_M^*)$ is spanned by
the total fiber (with multiplicities)
\[
\bfzero + \bfalpha + 2\bfgamma_0 + 2\bfgamma_1 + \cdots + 2\bfgamma_M + \bfbeta + \bfbeta'.
\]
This can be computed directly, or from the fact that all total
fibers are algebraically (and hence numerically) equivalent.
We define a reduced matrix
\[
A(I_M^*)^\red = \text{the matrix $A(I_M^*)$ with the first row and column deleted.}
\]
The matrix~$A(I_M^*)^\red$ is invertible, and indeed one can check that
\[
\det\bigl(A(I_M^*)^\red\bigr)=(-1)^M\cdot4.
\]
\par
We define a vector whose entries are the non-$\bfzero$ fibral components and a
vector whose entries are the coefficients of~$\F_P$:
\begin{align*}
  \bfF &= (\bfalpha,\bfgamma_0,\bfgamma_1,\ldots,\bfgamma_M,\bfbeta,\bfbeta'), \\
  \tilde\F_P &= (a,c_0,c_1,\ldots,c_M,b,b'),
\end{align*}
so
\[
\F_P=\bfF\cdot{^t\tilde\F_P}
\quad\text{and}\quad
A(I_M^*)^\red = {^t\!\bfF}\cdot\bfF.
\]
The $v$-fibral part of the (normalized) divisor~$\F_P$ is determined by
\[
P\cdot {^t\!\bfF}
= A(I_M^*)^\red\cdot {^t\tilde\F_P},
\quad\text{and thus}\quad
{^t\tilde\F_P} = \bigl(A(I_M^*)^\red\bigr)^{-1}\cdot \bigl(P\cdot {^t\!\bfF}\bigr).
\]
Since we have normalized~$\F_P$ so that~$i_v\bigl(\F_P,(O)\bigr)=0$, the
quantity appearing in~\eqref{eqn:lhatvPivPOivFPPO} is
\[
i_v\bigl(\F_P,(P)\bigr)
= P\cdot\F_P
= P\cdot\bfF\cdot{^t\tilde\F_P}
= (P\cdot\bfF)\cdot\bigl(A(I_M^*)^\red\bigr)^{-1}\cdot \bigl(P\cdot {^t\!\bfF}\bigr).
\]
\par
Figure~\ref{figure:IMstarintersectionmatrixinverse} gives an explicit
formula for~$\bigl(A(I_M^*)^\red\bigr)^{-1}$.  But we really only need
a few of its entries.

\begin{figure}[t]
\[
\bigl(A(I_M^*)^\red\bigr)^{-1}
= -\left(
\begin{array}{@{}ccccc|cc}
1 & 1 & 1 & \cdots & 1 & \frac12 & \frac12 \\[1\jot]
1 & 2 & 2 & \cdots & 2 & \frac22 & \frac22 \\[1\jot]
1 & 2 & 3 & \cdots & 3 & \frac32 & \frac32 \\[1\jot]
& \vdots & & \ddots & \vdots & \vdots & \vdots \\
1 & 2 & 3 & \cdots & M+2 & \frac{M+2}{2} & \frac{M+2}{2} \\[1\jot] \hline
\frac12 & \frac22 & \frac32 & \cdots & \frac{M+2}{2}  & \frac{M+4}{2}  & \frac{M+2}{2}
\vphantom{{\frac11}^1}  \\[1\jot]
\frac12 & \frac22 & \frac32 & \cdots & \frac{M+2}{2}  & \frac{M+2}{2}  & \frac{M+4}{2}  \\
\end{array}
\right)
\]
\caption{The inverse of the reduced  intersection matrix  for $I_M^*$ reduction.}
\label{figure:IMstarintersectionmatrixinverse}
\end{figure}

Indeed, since~$P$ must go through a multiplicity-1
component and since we are assuming that
$P\notin\Ecal^{(\bfzero)}(\FF_v)$, there are only three cases:
\[
P\cdot {^t\!\bfF} = \begin{cases}
  ^t(1,0,0,\ldots,0,0,0) &\text{if $P\in\Ecal^{(\bfalpha)}(\FF_v)$,} \\
  ^t(0,0,0,\ldots,0,1,0) &\text{if $P\in\Ecal^{(\bfbeta)}(\FF_v)$,} \\
  ^t(0,0,0,\ldots,0,0,1) &\text{if $P\in\Ecal^{(\bfbeta')}(\FF_v)$.} \\
  \end{cases}
\]
Writing~$B_{i,j}$ for the~$(i,j)$th entry of a matrix~$B$,
Figure~\ref{figure:IMstarintersectionmatrixinverse} gives the formulas
\[
i_v\bigl(\F_P,(P)\bigr) = \begin{cases}
  \bigl(A(I_M^*)^\red\bigr)^{-1}_{1,1} = -1 &\text{if $P\cdot\bfalpha=1$,} \\
  \bigl(A(I_M^*)^\red\bigr)^{-1}_{M+3,M+3} = -\frac{M+4}{4} &\text{if $P\cdot\bfbeta=1$,} \\
  \bigl(A(I_M^*)^\red\bigr)^{-1}_{M+4,M+4} = -\frac{M+4}{4}&\text{if $P\cdot\bfbeta'=1$.} \\
  \end{cases}
\]
Hence
\begin{align}
  P\in\Ecal^{(\bfalpha)}(\FF_v)
  &\quad\Longrightarrow\quad
  \begin{aligned}[t]
  \hhat_\Ecal(P)
  &= \frac{1}{12}v(\Dcal_\Ecal) - \frac12\\
  &= \frac{1}{12}v(j_E^{-1}), \\
  \end{aligned} \\
  P\in\Ecal^{(\bfbeta)}(\FF_v)
  &\quad\Longrightarrow\quad
  \begin{aligned}[t]
  \hhat_\Ecal(P)
  &= \frac{1}{12}v(\Dcal_\Ecal) - \frac{M+4}{8}\\
  &= -\frac{1}{24}v(j_E^{-1}) + \frac{1}{8}\bigl(v(j_E^{-1}) - M\bigr),\\
  \end{aligned}
\end{align}
where we have used the fact that
\[
v(\Dcal_\Ecal) = v(j_E^{-1}) + 6\quad\text{for $I_M^*$ reduction.}
\]
These formulas hold in all characteristics. If we further assume that
the residue characteristic is not~$2$, then~$v(j_E^{-1})=M$, and
the second formula simplifies.

\begin{acknowledgement}
The author would like to thank David Masser for suggesting revisiting
this problem in the function field setting, Pascal Autissier for
sharpening the inequality in Lemma~\ref{lemma:hhatineq}, Jordan
Ellenberg for a helpful observation, Doug Ulmer for help
regarding inseparable extensions, and the referee for numerous useful
suggestions. Various numerical and algebraic calculations were done
using Pari-GP~\cite{PariGP}.  The author's research was partially
supported by Simons Collaboration Grant \#712332.
\end{acknowledgement}


\end{document}